\documentclass[10pt]{amsart}
\usepackage{amssymb,amsbsy,amsmath,amsfonts,times, amscd, color}
\usepackage{latexsym,euscript,exscale}

\usepackage{helvet}

\title{Large scale geometry of automorphism groups}

\author {Christian Rosendal}
\address{Department of Mathematics, Statistics, and Computer Science (M/C 249)\\University of Illinois at Chicago\\851 S. Morgan St.\\Chicago, IL 60607-7045\\USA}
\email{rosendal.math@gmail.com}
\urladdr{http://homepages.math.uic.edu/$~$rosendal}

\date {}
\newcommand{\norm}[1]{\lVert#1\rVert}

\newcommand {\N}{\mathbb N}
\newcommand {\OO}{\mathcal O}
\newcommand {\Q}{\mathbb Q}

\newcommand {\Z}{\mathbb Z}
\newcommand {\C}{\mathbb C}
\newcommand {\U}{\mathbb U}

\newcommand{\om}{\omega}

\newcommand{\tom} {\emptyset}

\newcommand{\inj}{\hookrightarrow}

\newcommand{\saa}{\Rightarrow}
\newcommand{\equi}{\Longleftrightarrow}

\newcommand{\til}{\rightarrow}

\newcommand {\Del}{ \; \Big| \;}
\newcommand {\del}{ \; \big| \;}

\newcommand {\ku} {\mathcal}

\newcommand{\ov}{\overline}
\newcommand{\inv}{^{-1}}

\newcommand {\e} {\exists}
\renewcommand {\a} {\forall}

\newcommand{\forkindep}[1][]{\mathop{\mathop{\vcenter{\hbox{\oalign{\noalign{\kern-.3ex}
\hfil$\vert$\hfil\cr\noalign{\kern-.7ex}$\smile$\cr\noalign{\kern-.3ex}}}}}\displaylimits_{#1}}}

\newtheorem{thm}{Theorem}
\newtheorem{cor}[thm]{Corollary}
\newtheorem{lemme}[thm]{Lemma}
\newtheorem{prop} [thm] {Proposition}
\newtheorem{defi} [thm] {Definition}

\newtheorem{prob}[thm]{Problem}

\theoremstyle{definition}

\newtheorem{exa}[thm]{Example}

\definecolor{groen}{rgb}{0,0.5,.7}
\definecolor{gul}{rgb}{0.94,0.8,0}
\definecolor{blaa}{rgb}{0.16,0,0.6}
\definecolor{roed}{rgb}{1,0,0}

\usepackage[sc]{mathpazo}

\begin{document}

\thanks{The author was partially supported by a grant from the Simons Foundation (Grant
\#229959) and also recognises support from the NSF (DMS 1201295).}

\maketitle

\tableofcontents

\section{Introduction}
The present paper constitutes the third part of a study of the large scale geometry of metrisable group, the first two parts appearing in \cite{large scale geom}. Whereas the first part provided the foundations for this and the second part studied affine isometric actions on Banach spaces, we here make a deeper investigation of the special case of automorphism groups of countable first-order model theoretical structures.

\subsection{Non-Archimedean Polish groups and large scale geometry}
The automorphism groups under consideration here are the automorphism groups ${\rm Aut}(\bf M)$ of countable first-order model theoretical structures. Let us first recall that a {\em Polish group} is a separable topological group, whose topology can be induced by a complete metric. Such a group $G$ is moreover said to be {\em non-Archimedean} if there is a neighbourhood basis at the identity consisting of open subgroups of $G$. 

One particular source of examples of non-Archimedean Polish groups are first-order model theoretical structures. Namely, if $\bf M$ is a countable first-order structure, e.g., a graph, a group, a field or a lattice, we equip its automorphism group ${\rm Aut}(\bf M)$ with the {\em permutation group topology}, which is the group topology obtained by declaring the pointwise stabilisers 
$$
V_A=\{g\in {\rm Aut}({\bf M})\del \a x\in A\;\;g(x)=x\}
$$
of all finite subsets $A\subseteq \bf M$ to be open. In this case, one sees that a basis for the topology on ${\rm Aut}(\bf M)$ is given by the family of cosets $fV_A$, where $f\in {\rm Aut}(\bf M)$ and $A\subseteq \bf M$ is finite.

Now, conversely, if $G$ is a non-Archimedean Polish group, then, by considering its action on the left-coset spaces $G/V$, where $V$ varies over open subgroups of $G$, one can show that $G$ is topologically isomorphic to the automorphism group ${\rm Aut}(\bf M)$ of some first order structure $\bf M$.

The investigation of non-Archimedean Polish groups via the interplay between the model theoretical properties of the structure $\bf M$ and the dynamical and topological properties of the automorphism group ${\rm Aut}(\bf M)$ is currently very active as witnessed, e.g.,  by the papers \cite{kpt, turbulence, tsankov, wap}. Our goal here is to consider another class of structural properties of ${\rm Aut}(\bf M)$ by applying the large scale geometrical methods developed in our companion paper \cite{large scale geom}. 

For this, we need to recall the fundamental notions and some of the main results of \cite{large scale geom}. So, in the following, fix a separable metrisable topological group $G$.
A subset $A\subseteq G$ is said to have {\em property (OB) relative to $G$} if $A$ has finite diameter with respect to every compatible left-invariant metric on $G$ (the existence of compatible left-invariant metrics being garanteed by the Birkhoff--Kakutani theorem). Also, $G$ is said to have {\em property (OB)} if it has property (OB) relative to itself and the {\em local property (OB)} if there is a neighbourhood $U\ni 1$ with property (OB) relative to $G$.  As it turns out,   the relative property (OB)  for $A$ is equivalent to the requirement that, for all neighbourhoods $V\ni 1$, there are a finite set $F\subseteq G$ and a $k\geqslant 1$ so that $A\subseteq (FV)^k$. 

A compatible left-invariant metric $d$ on $G$ is {\em metrically proper} if the class of $d$-bounded sets coincides with the sets having property (OB) relative to $G$. Moreover, $G$ admits a compatible metrically proper left-invariant metric if and only if it has the local property (OB).  Also, an isometric action $G\curvearrowright (X,d_X)$ on a metric space is {\em metrically proper} if, for all $x$ and $K$, the set
$\{g\in G\del d_X(gx,x)\leqslant K\}$ has property (OB) relative to $G$.

We can order the set of compatible left-invariant metrics on $G$ by setting $\partial\lesssim d$  if there is a constant $K$ so that $\partial\leqslant K\cdot d+K$. We then define a metric $d$ to be {\em maximal} if it maximal with respect to this ordering. A maximal metric is always metrically proper and conversely, if $d$ is metrically proper, then $d$ is maximal if and only if $(G,d)$ is {\em large scale geodesic}, meaning that there is a constant $K$ so that, for all $g,f\in G$, one can find a finite path $h_0=g, h_1, \ldots, h_n=f$ with $d(h_{i-1},h_{1})\leqslant K$ and $d(g,f)\leqslant K\cdot \sum_{i=1}^nd(h_{i-1}, h_i)$. Now, a maximal metric, if such exist, is unique up to quasi-isometry, where a map $F\colon (X,d_X)\til (Y, d_Y)$ between metric spaces is a {\em quasi-isometry} if there are constants $K,C$ so that 
$$
\frac 1Kd_X(x,y)-C\leqslant d_Y(Fx, Fy)\leqslant K\cdot d_X(x,y)+C
$$ 
and $F[X]$ is a $C$-net in $Y$. 

The group $G$ admits a maximal compatible left-invariant metric if and only if $G$ is generated by an open set  with property (OB) relative to $G$. So, such $G$ have a uniquely defined {\em quasi-isometric type} given by any maximal metric. Finally, we have a version of the \v{S}varc--Milnor lemma, namely, if $G\curvearrowright (X,d_X)$ is a transitive metrically proper continuous isometric action on a large scale geodesic metric space, then $G$ has a maximal metric and, for every $x\in X$, the mapping 
$$
g\in G\mapsto g(x)\in X
$$
is a quasi-isometry. Of course, one of the motivations for developing this theory is that the  quasi-isometry type, whenever defined, is an isomorphic invariant of the group $G$.


\subsection{Main results}In the following, $\bf M$ denotes a countable first-order structure.
We use $\ov a, \ov b, \ov c, \ldots$  as variables for finite tuples of elements of $\bf M$ and shall write $(\ov a, \ov b)$ to denote the concatenation of the tuples $\ov a $ and $\ov b$. The automorphism group ${\rm Aut}(\bf M)$ acts naturally on  tuples $\ov a=(a_1,\ldots, a_n)\in {\bf M}^n$ via 
$$
g\cdot (a_1, \ldots, a_n)=(ga_1, \ldots, ga_n).
$$
With this notation, the pointwise stabiliser subgroups $V_{\ov a}=\{g\in {\rm Aut}({\bf M})\del g\cdot \ov a=\ov a\}$, where $\ov a$ ranges over all finite tuples in $\bf M$, form a neighbourhood basis at the identity in ${\rm Aut}(\bf M)$. So, if $A\subseteq \bf M$ is the finite set enumerated by $\ov a$ and $\bf A\subseteq \bf M$ is the substructure generated by $A$, we have $V_{\bf A}=V_A=V_{\ov a}$.
An {\em orbital type} $\ku O$ in $\bf M$ is simply the orbit of some tuple $\ov a$ under the action of ${\rm Aut}({\bf M})$.  Also, we let $\OO(\ov a)$ denote the orbital type of $\ov a$ , i.e., $\OO(\ov a)={\rm Aut}({\bf M})\cdot \ov a$.

Our first results provides a necessary and sufficient criterion for when an automorphism group has a well-defined quasi-isometric type and associated tools allowing for concrete computations of this same type. 
\begin{thm}\label{quasiisometry intro}
The automorphism group  ${\rm Aut}({\bf M})$ admits a maximal compatible left-invariant metric if and only if there  is a tuple $\ov a$ in ${\bf M}$ satisfying the following two requirements
\begin{enumerate}
\item for every $\ov b$, there is a finite family $\ku S$ of orbital types and an $n\geqslant 1$ such that whenever $(\ov a, \ov c)\in \OO(\ov a, \ov b)$ there is a path $\ov d_0, \ldots, \ov d_n\in  \OO(\ov a, \ov b)$ with $\ov d_0=(\ov a, \ov b)$, $\ov d_n=(\ov a, \ov c)$ and $\ku O(\ov d_i,\ov d_{i+1})\in \ku S$ for all $i$,
\item there is a finite family $\ku R$ of orbital types so that, for all $\ov b\in \OO(\ov a)$, there is a finite path $\ov d_0, \ldots, \ov d_m\in \OO(\ov a)$ with $\ov d_0=\ov a$, $\ov d_m=\ov b$ and 
$\ku O(\ov d_i, \ov d_{i+1})\in  \ku R$ for all $i$. 
\end{enumerate}
Moreover, suppose $\ov a$ and $\ku R$ are as above and  let $\mathbb X_{\ov a, \ku R}$ be the connected graph on the set of vertices $\ku O(\ov a)$ with edge relation
$$
(\ov b, \ov c)\in {\rm Edge}\; \mathbb X_{\ov a, \ku R} \;\equi \; \ov b\neq \ov c\;\&\;\Big(\ku O(\ov b, \ov c)\in \ku R \;\text{ or }\;\ku O(\ov c, \ov b)\in \ku R\Big).
$$
Then the mapping 
$$
g\in {\rm Aut}({\bf M})\mapsto g\cdot \ov a \in \mathbb X_{\ov a, \ku R}
$$
is a quasi-isometry.
\end{thm}
Moreover, condition (1) alone gives a necessary and sufficient criterion for ${\rm Aut}(\bf M)$ having the local property (OB) and thus a metrically proper metric.

With this at hand, we can subsequently relate the properties of the theory $T={\rm Th}(\bf M)$ of the model $\bf M$ with properties of its automorphism group and, among other results, use this to construct affine isometric actions on Banach spaces.
\begin{thm}
Suppose $\bf M$ is a countable atomic model of a stable theory $T$ so that ${\rm Aut}(\bf M)$ has the local property (OB). Then ${\rm Aut}(\bf M)$ admits a metrically proper continuous affine isometric action on a reflexive Banach space. 
\end{thm}

Though Theorem \ref{quasiisometry intro} furnishes an equivalent reformulation of admitting a metrically proper or maximal compatible left-invariant metric, it is often useful to have more concrete instances of this. A particular case of this, is when $\bf M$ admits an obital $A$-independence relation, $\forkindep[A]$, that is, an independence relation over a finite subset $A\subseteq \bf M$ satisfying the usual properties of symmetry, monotonicity, existence and stationarity (see Definition \ref{indep rel} for a precise rendering).
In particular, this applies to  the Boolean algebra of clopen subsets of Cantor space with the dyadic probability measure and to the $\aleph_0$-regular tree.
\begin{thm}
Suppose $A$ is a finite subset of a countable structure $\bf M$ and $\forkindep[A]$ an orbital $A$-independence relation. Then the pointwise stabiliser subgroup $V_A$ has property (OB). Thus, if $A=\tom$, the automorphism group ${\rm Aut}(\bf M)$ has  property (OB) and, if $A\neq \tom$, ${\rm Aut}(\bf M)$ has the local property (OB). 
\end{thm}

Now, model theoretical independence relations arise, in particular, in models of $\omega$-stable theories. Though stationarity of the independence relation may fail, we nevertheless arrive at the following result.

\begin{thm}
Suppose that $\bf M$ is a saturated countable model of an $\om$-stable theory. Then ${\rm Aut}(\bf M)$ has property (OB).
\end{thm}
In this connection, we should mention an earlier observation by P. Cameron, namely, that automorphism groups of countable $\aleph_0$-categorical structures have property (OB). Also, it remains an open problem if there are similar results valid for countable atomic models of $\omega$-stable theories.
\begin{prob}
Suppose $\bf M$ is a countable atomic model of an $\om$-stable theory. Does ${\rm Aut}(\bf M)$ have the local property (OB)?
\end{prob}

A particular setting giving rise to orbital independence relations, which has earlier been studied by K. Tent and M. Ziegler \cite{tent}, is Fra\"iss\'e classes admitting a canonical amalgamation construction. For our purposes, we need a stronger notion than that considered in \cite{tent} and say that a Fra\"iss\'e class $\ku K$ admits a functorial amalgamation over some $\bf A\in \ku K$ if there is a map $\Theta$ that to every pair of embeddings $\eta_1\colon {\bf A}\inj {\bf B}_1$ and  $\eta_2\colon {\bf A}\inj {\bf B}_2$ with ${\bf B}_i\in \ku K$ produces an amalgamation of ${\bf B}_1$ and ${\bf B}_2$ over these embeddings so that $\Theta$ is symmetric in its arguments and commutes with embeddings (see Definition \ref{funct amal} for full details). 

\begin{thm}
Suppose $\ku K$ is a Fra\"iss\'e class with limit $\bf K$ admitting a functorial amalgamation over some  $\bf A\in \ku  K$. Then ${\rm Aut}(\bf K)$ admits a  metrically proper compatible left-invariant metric. Moreover, if $\bf A$ is generated by the empty set, then ${\rm Aut}(\bf K)$ has property (OB).
\end{thm}

Finally, using the above mentioned results, we may compute quasi-isometry types of various groups, e.g., those having property (OB) being quasi-isometric to a point, while others have more interesting structure.
\begin{enumerate}
\item The automorphism group ${\rm Aut}(\bf T)$ of the $\aleph_0$-regular tree $\bf T$ is quasi-iso\-metric to $\bf T$.
\item The group ${\rm Aut}({\bf T},\mathfrak e)$ of automorphisms of $\bf T$ fixing an end $\mathfrak e$ is quasi-isometric to $\bf T$.
\item The isometry group ${\rm Isom}(\Q\U)$ of the rational Urysohn metric space $\Q\U$ is quasi-isometric to $\Q\U$.
\end{enumerate}

\section{Orbital types formulation}
If $G$ is a non-Archimedean Polish group admitting a maximal compatible left-invariant metric, there is a completely abstract way of identifying its quasi-isometry type. Namely, fix a symmetric open set $U\ni 1$ generating the group and having property (OB) relative to $G$. Analogous to a construction of H. Abels (Beispiel 5.2 \cite{abels}) for locally compact totally disconnected groups, we construct a vertex transitive and metrically proper  action on a countable connected graph as follows. First, pick an open subgroup $V$ contained in $U$ and let $A\subseteq G$ be a countable  set so that $VUV=AV$. Since $a\in VUV$ implies that also $a\inv \in (VUV)\inv=VUV$, we may assume that $A$ is symmetric, whereby
$AV=VUV=(VUV)\inv=V\inv A\inv=VA$ and thus also  $(VUV)^k=(AV)^k=A^kV^k=A^kV$ for all $k\geqslant 1$. In particular, we note that $A^kV$ has property (OB) relative to $G$ for all $k\geqslant 1$. 

The graph $\mathbb X$ is now defined to be the set $G/V$ of left-cosets of $V$ along with the set of edges $\big\{\{gV,gaV\}\del a\in A\;\&\; g\in G\big\}$. Note that the left-multiplication action of $G$ on $G/V$ is a vertex transitive action of $G$ by automorphisms of $\mathbb X$.
Moreover, since $G=\bigcup_k(VUV)^k=\bigcup_kA^kV$, one sees that the graph $\mathbb X$ is connected and hence the shortest path distance $\rho$ is a well-defined metric on $\mathbb X$.
    
We claim that the action $G\curvearrowright \mathbb X$ is metrically proper. Indeed, note that, if  $g_n\til \infty$ in $G$, then $(g_n)$ eventually leaves every set with property (OB) relative to $G$ and thus, in particular, leaves every $A^kV$. Since, the $k$-ball around the vertex $1V\in \mathbb X$ is contained in the set $A^kV$, one sees that $\rho(g_n\cdot 1V,1V)\til \infty$, showing that the action is metrically proper.
Therefore, by our version of the \v{S}varc--Milnor Lemma, the mapping $g\mapsto  gV$ is a quasi-isometry between  $G$ and $(\mathbb X,\rho)$.

However, this construction neither addresses the question of when $G$ admits a maximal metric nor provides a very informative manner of defining this. For those questions, we need to investigate matters further. In the following, $\bf M$ will be a fixed countable first-order structure.

\begin{lemme}\label{S til F}
Suppose $\overline a$ is a finite tuple in ${\bf M}$ and $\ku S$ is a finite family of  orbital types in ${\bf M}$.
Then there is a finite set $F\subseteq {\rm Aut}({\bf M})$ so that, whenever $\ov a_0, \ldots, \ov a_n\in \OO(\ov a)$, $\ov a_0=\ov a$ and $\ku O(\ov a_i, \ov a_{i+1})\in \ku S$ for all $i$, then $\ov a_n\in (V_{\ov a}F)^n\cdot \ov a$. 
\end{lemme}

\begin{proof}
For each orbital type $\OO\in \ku S$, pick if possible some $f\in {\rm Aut}({\bf M})$ so that $(\ov a, f\ov a)\in \OO$ and let $F$ be the finite set of these $f$. Now, suppose that $\ov a_0, \ldots, \ov a_n\in \OO(\ov a)$, $\ov a_0=\ov a$ and $\ku O(\ov a_i, \ov a_{i+1})\in \ku S$ for all $i$. Since the  $\ov a_i$ are orbit equivalent, we can inductively choose $h_1,\ldots, h_n\in {\rm Aut}({\bf M})$ so that $\ov a_i=h_1\cdots h_i\cdot \ov a$. Thus, for all $i$, we have $(\ov a, h_{i+1}\ov a)= (h_1\cdots h_i)\inv\cdot  (\ov a_i, \ov a_{i+1})$, whereby there is some $f\in F$ so that $(\ov a, h_{i+1}\ov a)\in \OO(\ov a, f\ov a)$. It follows that, for some $g\in {\rm Aut}({\bf M})$, we have $(g\ov a,gh_{i+1}\ov a)=(\ov a, f\ov a)$, i.e., $g\in V_{\ov a}$ and $f\inv gh_{i+1}\in V_{\ov a}$, whence also $h_{i+1}\in V_{\ov a}FV_{\ov a}$.
Therefore, $\ov a_n=h_1\cdots h_n\cdot \ov a\in (V_{\ov a}FV_{\ov a})^n\cdot \ov a=(V_{\ov a}F)^n\cdot \ov a$.
\end{proof}

\begin{lemme}\label{F til S}
Suppose $\overline a$ is a finite tuple in ${\bf M}$ and $F\subseteq {\rm Aut}({\bf M})$ is a finite set. Then there is a finite family $\ku S$ of orbital types  in ${\bf M}$ so that, for all $g\in (V_{\ov a}F)^n$, there are  $\ov a_0, \ldots, \ov a_n\in \OO(\ov a)$ with $\ov a_0=\ov a$ and $\ov a_n =g\ov a$   satisfying $\ku O(\ov a_i, \ov a_{i+1})\in \ku S$ for all $i$.
\end{lemme}

\begin{proof}
We let $\ku S$ be the collection of $\ku O(\ov a, f\ov a)$ with $f\in F$. Now suppose that $g\in(V_{\ov a}F)^n$ and write $g=h_1f_1\cdots h_nf_n$ for $h_i\in V_{\ov a}$ and $f_i\in F$. Setting $\ov a_i=h_1f_1\cdots h_if_i\cdot  \ov a$, we see that 
$$
(\ov a_i, \ov a_{i+1})=h_1f_1\cdots h_if_i h_{i+1}\cdot (h_{i+1}\inv \ov a, f_{i+1}\ov a)=h_1f_1\cdots h_if_i h_{i+1}\cdot (\ov a, f_{i+1}\ov a),
$$
i.e., $\ku O(\ov a_i, \ov a_{i+1})\in \ku S$ as required.
\end{proof}

\begin{lemme}\label{condition for OB}
Suppose $\overline a$ is a finite tuple in ${\bf M}$. Then the following are equivalent.
\begin{enumerate}
\item The pointwise stabiliser $V_{\ov a}$ has property (OB) relative to ${\rm Aut}({\bf M})$.
\item For every tuple $\ov b$ in ${\bf M}$, there is a finite family $\ku S$ of orbital types and an $n\geqslant 1$ so that whenever $(\ov a, \ov c)\in \OO(\ov a, \ov b)$ there is a path $\ov d_0, \ldots, \ov d_n\in  \OO(\ov a, \ov b)$ with $\ov d_0=(\ov a, \ov b)$, $\ov d_n=(\ov a, \ov c)$ and $\ku O(\ov d_i,\ov d_{i+1})\in \ku S$ for all $i$. 
\end{enumerate}
\end{lemme}

\begin{proof}
(1)$\saa$(2): Suppose that $V_{\ov a}$ has property (OB) relative to ${\rm Aut}({\bf M})$ and that $\ov b$ is a tuple in ${\bf M}$. This means that, for every neigbourhood $U\ni1$, there is a finite set $F\subseteq {\rm Aut}(\bf M)$ and an $n\geqslant 1$ so that $V_{\ov a}\subseteq (UF)^n$. In particular, this holds for $U=V_{(\ov a, \ov b)}$. Let now $\ku S$ be the finite family of orbital types associated to $F$  and the tuple $(\ov a, \ov b)$ as given in Lemma \ref{F til S}. Assume also that $(\ov a, \ov c)\in \OO(\ov a, \ov b)$, i.e., that $\ov c=g\cdot \ov b$ for some $g\in V_{\ov a}\subseteq (V_{(\ov a,\ov b)}F)^n$. By assumption on $\ku S$, there is a path $\ov d_0, \ldots, \ov d_n\in  \OO(\ov a, \ov b)$ with  and $\ku O(\ov d_i,\ov d_{i+1})\in \ku S$ for all $i$ satisfying $\ov d_0=(\ov a, \ov b)$, $\ov d_n=(\ov a, \ov c)$ as required.

(2)$\saa$(1): Assume that (2) holds. To see that $V_{\ov a}$ has property (OB) relative to ${\rm Aut}(\bf M)$, it is enough to verify that, for all tuples $\ov b$ in $\bf M$, there is a finite set $F\subseteq {\rm Aut}(\bf M)$ and an $n\geqslant 1$ so that $V_{\ov a}\subseteq (V_{(\ov a, \ov b)}FV_{(\ov a, \ov b)})^n$. But given $\ov b$, let $\ku S$ be as in (2) and pick a finite set $F\subseteq {\rm Aut}(\bf M)$ associated to $\ku S$ and the tuple $(\ov a, \ov b)$ as provided by Lemma \ref{S til F}. 

Now, if $g\in V_{\ov a}$, then $(\ov a, g\ov b)=g\cdot (\ov a, \ov b)\in \OO(\ov a, \ov b)$. So by (2) there is a path $\ov d_0, \ldots, \ov d_n\in  \OO(\ov a, \ov b)$ with $\ov d_0=(\ov a, \ov b)$, $\ov d_n=(\ov a, g\ov b)$ and $\ku O(\ov d_i,\ov d_{i+1})\in \ku S$ for all $i$. Again, by the choice of $F$, we have $g\cdot (\ov a,\ov b)=\ov d_n\in (V_{(\ov a, \ov b)}F)^n\cdot (\ov a, \ov b)$, whence $g\in (V_{(\ov a, \ov b)}F)^n\cdot V_{(\ov a, \ov b)}= (V_{(\ov a, \ov b)}FV_{(\ov a, \ov b)})^n$ as required.
\end{proof}

Using that the $V_{\ov a}$ form a neighbourhood basis at the identity in ${\rm Aut}(\bf M)$, we obtain the following criterion for the local property (OB) and hence the existence of a metrically proper metric.

\begin{thm}
The following are equivalent for the automorphism group ${\rm Aut}(\bf M)$ of a countable structure $\bf M$. 
\begin{enumerate}
\item There is a metrically proper compatible left-invariant metric on ${\rm Aut}(\bf M)$,
\item ${\rm Aut}(\bf M)$ has the local property (OB),
\item  there is a tuple $\ov a$ so that, for every $\ov b$, there is a finite family $\ku S$ of orbital types and an $n\geqslant 1$ such that whenever $(\ov a, \ov c)\in \OO(\ov a, \ov b)$ there is a path $\ov d_0, \ldots, \ov d_n\in  \OO(\ov a, \ov b)$ with $\ov d_0=(\ov a, \ov b)$, $\ov d_n=(\ov a, \ov c)$ and $\ku O(\ov d_i,\ov d_{i+1})\in \ku S$ for all $i$. 
\end{enumerate}
\end{thm}

Our goal is to be able to compute the actual quasi-isometry type of various automorphism groups and, for this, the following graph turns out to be of central importance.

\begin{defi}
Suppose $\ov a$ is a tuple in ${\bf M}$ and $\ku S$ is a finite family of orbital types in $\bf M$. We let $\mathbb X_{\ov a, \ku S}$ denote the graph whose vertex set is the orbital type $\OO(\ov a)$ and whose edge relation  is given by 
$$
(\ov b, \ov c)\in {\rm Edge}\; \mathbb X_{\ov a, \ku S} \;\equi \; \ov b\neq \ov c\;\&\;\Big(\ku O(\ov b, \ov c)\in \ku S \;\text{ or }\;\ku O(\ov c, \ov b)\in \ku S\Big).
$$
Let also $\rho_{\ov a, \ku S}$ denote the corresponding shortest path metric on $\mathbb X_{\ov a, \ku S}$, where we stipulate that $\rho_{\ov a, \ku S}(\ov b, \ov c)=\infty$ whenever $\ov b$ and $\ov c$ belong to distinct connected components.
\end{defi}

We remark that, as the vertex set of $\mathbb X_{\ov a, \ku S}$ is just the orbital type of $\ov a$, the automorphism group ${\rm Aut}(\bf M)$ acts transitively on the vertices of $\mathbb X_{\ov a, \ku S}$. Moreover,  the edge relation is clearly invariant, meaning that ${\rm Aut}(\bf M)$ acts vertex transitively by automorphisms on $\mathbb X_{\ov a, \ku S}$. In particular, ${\rm Aut}(\bf M)$ preserves $\rho_{\ov a, \ku S}$.

Note that $\rho_{\ov a, \ku S}$ is actual metric exactly when $\mathbb X_{\ov a, \ku S}$ is a connected graph. Our next task is to decide when this  happens.

\begin{lemme}
Suppose $\ov a$ is a tuple in ${\bf M}$. Then the following are equivalent.
\begin{enumerate}
\item ${\rm Aut}({\bf M})$ is finitely generated over $V_{\ov a}$,
\item  there is a finite family $\ku S$  of orbital types so that $\mathbb X_{\ov a, \ku S}$ is connected. 
\end{enumerate}
\end{lemme}

\begin{proof}
(1)$\saa$(2): Suppose that ${\rm Aut}({\bf M})$ is finitely generated over $V_{\ov a}$ and pick a finite set $F\subseteq {\rm Aut}(\bf M)$ containing $1$ so that ${\rm Aut}({\bf M})=\langle V_{\ov a}\cup F\rangle$. Let also $\ku S$ be the finite family of orbital types associated to $F$ and $\ov a$ as given by Lemma \ref{F til S}. To see that $\mathbb X_{\ov a, \ku S}$ is connected, let $\ov b\in \OO(\ov a)$ be any vertex and write $\ov b=g\ov a$ for some $g\in {\rm Aut}(\bf M)$. Find also $n\geqslant 1$ so that $g\in (V_{\ov a}F)^n$. By the choice of $\ku S$, it follows that there are $\ov c_0, \ldots, \ov c_n\in \OO(\ov a)$ with $\ov c_0=\ov a$, $\ov c_n=g\ov a$ and $\ku O(\ov c_i,\ov c_{i+1})\in \ku S$ for all $i$. Thus, $\ov c_0, \ldots, \ov c_n$ is a path from $\ov a$ to $\ov b$ in $\mathbb X_{\ov a, \ku S}$. Since every vertex is connected to $\ov a$, $\mathbb X_{\ov a, \ku S}$ is a connected graph.

(2)$\saa$(1): Assume $\ku S$ is a finite family  of orbital types so that $\mathbb X_{\ov a, \ku S}$ is connected. We let $\ku T$ consist of all orbital types $\OO(\ov b, \ov c)$ so that either $\OO(\ov b, \ov c)\in \ku S$ or $\OO( \ov c, \ov b)\in \ku S$ and note that $\ku T$ is also finite. Let also $F\subseteq {\rm Aut}(\bf M)$ be the finite set associated to $\ov a$ and $\ku T$ as given by Lemma \ref{F til S}. Then, if $g\in {\rm Aut}(\bf M)$, there is a path $\ov c_0, \ldots, \ov c_n$ in $\mathbb X_{\ov a, \ku S}$ from $\ov c_0=\ov a$ to $\ov c_n=g\ov a$, whence $\ku O(\ov c_i, \ov c_{i+1})\in \ku T$ for all $i$. By the choice of $F$, it follows that $g\ov a=\ov c_n\in (V_{\ov a}F)^n\cdot \ov a$ and hence that $g\in (V_{\ov a}F)^n\cdot V_{\ov a}$. Thus,  ${\rm Aut}({\bf M})=\langle V_{\ov a}\cup F\rangle$.
\end{proof}

\begin{lemme}\label{metrically proper model}
Suppose $\ov a$ is a tuple in ${\bf M}$ so that the pointwise stabiliser $V_{\ov a}$ has property (OB) relative to ${\rm Aut}({\bf M})$ and assume that $\ku S$ is a finite family of orbital types. Then, for all natural numbers $n$, the set 
$$
\{g\in {\rm Aut}({\bf M})\del \rho_{\ov a, \ku S}(\ov a, g\ov a)\leqslant n\}
$$
has property (OB) relative to ${\rm Aut}({\bf M})$.

In particular, if the graph $\mathbb X_{\ov a, \ku S}$ is connected, then the continuous isometric action 
$$
{\rm Aut}({\bf M})\curvearrowright \mathbb (\mathbb X_{\ov a, \ku S}, \rho_{\ov a, \ku S})
$$
is metrically proper.
\end{lemme}

\begin{proof}
Let $F\subseteq {\rm Aut}(\bf M)$ be the finite set associated to $\ov a$ and $\ku S$ as given by Lemma \ref{S til F}. Then, if $g\in {\rm Aut}(\bf M)$ is such that $\rho_{\ov a, \ku S}(\ov a, g\ov a)=m\leqslant n$, there is a path $\ov c_0, \ldots, \ov c_m$ in $\mathbb X_{\ov a, \ku S}$ with $\ov c_0=\ov a$ and $\ov c_m=g\ov a$. Thus, by the choice of $F$, we have that $g\ov a=\ov c_m\in (V_{\ov a}F)^m\cdot \ov a$, i.e., $g\in (V_{\ov a}F)^m\cdot V_{\ov a}=(V_{\ov a}FV_{\ov a})^m$. In other words,
$$
\{g\in {\rm Aut}({\bf M})\del \rho_{\ov a, \ku S}(\ov a, g\ov a)\leqslant n\}\subseteq \bigcup_{m\leqslant n}(V_{\ov a}FV_{\ov a})^m
$$
and the latter set had property (OB) relative to ${\rm Aut}(\bf M)$.
\end{proof}

With these preliminary results at hand, we can now give a full characterisation of when an automorphism group ${\rm Aut}(\bf M)$ carries a well-defined large scale geometry and, moreover, provide a direct compution of this.

\begin{thm}\label{quasiisometry}
Let ${\bf M}$ be a countable structure. Then  ${\rm Aut}({\bf M})$ admits a maximal compatible left-invariant metric if and only if there  is a tuple $\ov a$ in ${\bf M}$ satisfying the following two requirements
\begin{enumerate}
\item for every $\ov b$, there is a finite family $\ku S$ of orbital types and an $n\geqslant 1$ such that whenever $(\ov a, \ov c)\in \OO(\ov a, \ov b)$ there is a path $\ov d_0, \ldots, \ov d_n\in  \OO(\ov a, \ov b)$ with $\ov d_0=(\ov a, \ov b)$, $\ov d_n=(\ov a, \ov c)$ and $\ku O(\ov d_i,\ov d_{i+1})\in \ku S$ for all $i$,
\item there is a finite family $\ku R$ of orbital types so that, for all $\ov b\in \OO(\ov a)$, there is a finite path $\ov d_0, \ldots, \ov d_m\in \OO(\ov a)$ with $\ov d_0=\ov a$, $\ov d_m=\ov b$ and 
$\ku O(\ov d_i, \ov d_{i+1})\in  \ku R$ for all $i$. 
\end{enumerate}
Moreover, if $\ov a$ and $\ku R$ are as in (2), then the mapping 
$$
g\in {\rm Aut}({\bf M})\mapsto g\cdot \ov a \in \mathbb X_{\ov a, \ku R}
$$
is a quasi-isometry between $ {\rm Aut}({\bf M})$ and $(\mathbb X_{\ov a, \ku R}, \rho_{\ov a, \ku R})$.
\end{thm}

\begin{proof}
Note that (1) is simply a restatement of $V_{\ov a}$ having property (OB) relative to ${\rm Aut}(\bf M)$, while (2) states that  there is $\ku R$ so that the graph $\mathbb X_{\ov a, \ku R}$ is connected, i.e., that ${\rm Aut}(\bf M)$ is finitely generated over $V_{\ov a}$. Together, these two properties are equivalent to the existence of a compatible left-invariant metric maximal for large distances.

For the moreover part, note that, as $\mathbb X_{\ov a, \ku R}$ is a connected graph, the metric space $ (\mathbb X_{\ov a, \ku R}, \rho_{\ov a, \ku R})$ is large scale geodesic. Thus, as the continuous isometric action ${\rm Aut}({\bf M})\curvearrowright \mathbb (\mathbb X_{\ov a, \ku S}, \rho_{\ov a, \ku S})$ is transitive (hence cobounded) and metrically proper, it follows from the \v{S}varc--Milnor lemma that  
$$
g\in {\rm Aut}({\bf M})\mapsto g\cdot \ov a \in \mathbb X_{\ov a, \ku R}
$$
is a quasi-isometry between $ {\rm Aut}({\bf M})$ and $(\mathbb X_{\ov a, \ku R}, \rho_{\ov a, \ku R})$.
\end{proof}

In cases where ${\rm Aut}(\bf M)$ may not admit a compatible left-invariant maximal metric, but only a metrically proper metric, it is still useful to have an explicit calculation of this. For this, the following lemma will be useful.
\begin{lemme}
Suppose $\ov a$ is finite tuple in a countable structure $\bf M$. Let also $\ku R_1\subseteq \ku R_2\subseteq \ku R_3\subseteq \ldots$ be an exhaustive sequence of finite sets of orbital types on $\bf M$ and define a metric $\rho_{\ov a, (\ku R_n)}$ on $\ku O(\ov a)$ by
\[\begin{split}
\rho_{\ov a, (\ku R_n)}(\ov b,\ov c)=&\\
\min\Big(\sum_{i=1}^kn_i\cdot & \rho_{\ov a,\ku R_{n_i}}(\ov d_{i-1},\ov d_i)\Del n_i\in \N\;\;\&\;\; \ov d_i\in \ku O(\ov a)\;\;\&\;\;  \ov d_0=\ov b\;\;\&\;\; \ov d_k=\ov c\Big).
\end{split}\]
Assuming that $V_{\ov a}$ has property (OB) relative to ${\rm Aut}(\bf M)$, then the isometric action ${\rm Aut}({\bf M})\curvearrowright \big(\ku O(\ov a),\rho_{\ov a, (\ku R_n)}\big)$ is metrically proper.
\end{lemme}

\begin{proof}
Let us first note that, since  the sequence $(\ku R_n)$ is exhaustive, every orbital type $\ku O(\ov b, \ov c)$ eventually belongs to some $\ku R_n$, whereby $\rho_{\ov a, (\ku R_n)}(\ov b,\ov c)$ is finite. Also, $\rho_{\ov a, (\ku R_n)}$ satisfies the triangle inequality by definition and hence is a metric. 

Note now that, since the $\ku R_n$ are increasing with $n$, we have, for all $m\in \N$,
$$
\rho_{\ov a, (\ku R_n)}(\ov b,\ov c)\leqslant m\;\;\saa\;\; \rho_{\ov a, \ku R_m}(\ov b,\ov c)\leqslant m
$$
and thus 
$$
\{g\in {\rm Aut}({\bf M})\del \rho_{\ov a, (\ku R_n)}(\ov a,g\ov a)\leqslant m\}\subseteq \{g\in {\rm Aut}({\bf M})\del \rho_{\ov a, \ku R_m}(\ov a,g\ov a)\leqslant m\}.
$$
By Lemma \ref{metrically proper model}, the latter set has property (OB) relative to ${\rm Aut}(\bf M)$, so the action ${\rm Aut}({\bf M})\curvearrowright \big(\ku O(\ov a),\rho_{\ov a, (\ku R_n)}\big)$ is metrically proper.
\end{proof}


\section{Homogeneous and atomic models}

\subsection{Definability of metrics}
Whereas the preceding sections have largely concentrated on the automorphism group ${\rm Aut}(\bf M)$ of a countable structure $\bf M$ without much regard to the actual structure $\bf M$, its language $\ku L$ or its theory $T={\rm Th}(\bf M)$, in the present section, we shall study how the theory $T$ may directly influence the large scale geometry of ${\rm Aut}(\bf M)$.

We recall that a structure $\bf M$ is {\em $\omega$-homogeneous} if, for all finite tuples $\ov a$ and $\ov b$ in $\bf M$ with the same type ${\rm tp}^{\bf M}(\ov a)={\rm tp}^{\bf M}(\ov  b)$ and all $c$ in $\bf M$, there is some $d$ in $\bf M$ so that ${\rm tp}^{\bf M}(\ov a,c)={\rm tp}^{\bf M}(\ov  b, d)$. By a back and forth construction, one sees that, in case $\bf M$ is countable, $\omega$-homogeneity is equivalent to the  condition
$$
{\rm tp}^{\bf M}(\ov a)={\rm tp}^{\bf M}(\ov  b)\;\;\equi\;\; \ku O(\ov a)=\ku O(\ov b).
$$
In other words, every orbital type $\ku O(\ov a)$ is type $\tom$-definable, i.e., type definable without parameters.

For a stronger notion, we say that $\bf M$ is {\em ultrahomogeneous} if it satisfies
$$
{\rm qftp}^{\bf M}(\ov a)={\rm qftp}^{\bf M}(\ov  b)\;\;\equi\;\; \ku O(\ov a)=\ku O(\ov b),
$$
where ${\rm qftp}^{\bf M}(\ov a)$ denotes the quantifier-free type of $\ov a$.  In other words, every orbital type $\ku O(\ov a)$ is defined by the quantifier-free type ${\rm qftp}^{\bf M}(\ov a)$.

Another  requirement is to demand that each individual orbital type $\ku O(\ov a)$ is $\tom$-definable in $\bf M$, i.e., definable by a single formula $\phi(\ov x)$ without parameters, that is, so that $\ov b\in \ku O(\ov a)$ if and only if ${\bf M}\models\phi(\ov b)$. We note that such a $\phi$ necessarily isolates the type ${\rm tp}^{\bf M}(\ov a)$. Indeed, suppose $\psi\in {\rm tp}^{\bf M}(\ov a)$. Then, if ${\bf M}\models \phi(\ov b)$, we have   $\ov b\in \ku O(\ov a)$ and thus also ${\bf M}\models \psi(\ov b)$, showing that ${\bf M}\models\a \ov x\; ( \phi\til \psi)$. Conversely, suppose $\bf M$ is a countable  $\omega$-homogeneous structure and $\phi(\ov x)$ is a formula without parameters isolating some type ${\rm tp}^{\bf M}(\ov a)$. Then, if ${\bf M}\models \phi(\ov b)$, we have ${\rm tp}^{\bf M}(\ov a)={\rm tp}^{\bf M}(\ov  b)$ and thus, by $\omega$-homogeneity,  $\ov b\in \ku O(\ov a)$.

We recall that a model $\bf M$ is {\em atomic} if every type realised in $\bf M$ is isolated. As is easy to verify (see Lemma 4.2.14 \cite{marker}), countable atomic models are $\omega$-homogeneous. So, by the discussion above,   we see that a countable model $\bf M$ is atomic if and only if every orbital type $\ku O(\ov a)$ is $\tom$-definable.

For example, if $\bf M$ is a locally finite ultrahomogeneous structure in a finite language $\ku L$, then $\bf M$ is atomic. This follows from the fact that if $\bf A$ is a finite structure in a finite language, then its isomorphism type is described by a single quantifier-free formula.

\begin{lemme}\label{definability}
Suppose $\ov a$ is a finite tuple in a countable atomic model $\bf M$.  Let $\ku S$ be a finite collection of orbital types in $\bf M$ and  $\rho_{\ov a, \ku S}$ denote the corresponding shortest path metric on $\mathbb X_{\ov a, \ku S}$.
Then, for every $n\in \N$, the relation $\rho_{\ov a, \ku S}(\ov b, \ov c)\leqslant n$ is $\tom$-definable in $\bf M$. 

Now, suppose instead that $\ku S_1\subseteq \ku S_2\subseteq \ldots$ is an exhaustive sequence of finite sets of orbital types on $\bf M$. Then, for every $n\in \N$, the relation $\rho_{\ov a, (\ku S_m)}(\ov b, \ov c)\leqslant n$ is similarly  $\tom$-definable in $\bf M$.
\end{lemme}

\begin{proof}Without loss of generality, every orbital type in $\ku S$ is of the form $\ku O(\ov b, \ov c)$, where $\ov b, \ov c\in \ku O(\ov a)$. Moreover, for such $\ku O(\ov b, \ov c)\in \ku S$, we may suppose that also $\ku O(\ov c, \ov b)\in \ku S$. 
Let now  $\phi_1(\ov x, \ov y), \ldots, \phi_k(\ov x, \ov y)$ be formulas without parameters defining the orbital types in $\ku S$. Then
\[\begin{split}
\rho_{\ov a, \ku S}(\ov b, \ov c)&\leqslant n \;\;\equi \\
{\bf M}\models\;\;&\bigvee_{m=0}^n\;\e \ov y_0,\ldots, \ov y_{m}\; \Big(\bigwedge_{j=0}^{m-1}\bigvee_{i=1}^k\phi_i(\ov y_j, \ov y_{j+1})\;\&\;  \ov b=\ov y_0\;\&\; \ov c=\ov y_m\Big),
\end{split}\]
showing that $\rho_{\ov a, \ku S}(\ov b, \ov c)\leqslant n$ is $\tom$-definable in $\bf M$. 

For the second case, pick formulas $\phi_{m,n}(\ov x,\ov y)$ without parameters defining the relations
$\rho_{\ov a, \ku S_m}(\ov b, \ov c)\leqslant n$ in $\bf M$. Then
\[\begin{split}
\rho_{\ov a, (\ku S_m)}&(\ov b, \ov c)\leqslant n\;\;\equi    \\
{\bf M}\models 
\bigvee_{k=0}^n&
\e \ov x_0, \ldots, \ov x_k 
\Big( 
\ov x_0=\ov b\;\&\;
\ov x_k=\ov c\;\&
\bigvee
\big\{   \bigwedge_{i=1}^k  \phi_{m_i, n_i}(\ov x_{i-1},\ov x_i)     \del       \sum_{i=1}^k m_i\cdot n_i\leqslant n       \big\}
\Big),
\end{split}\]
showing that $\rho_{\ov a, (\ku S_m)}(\ov b, \ov c)\leqslant n$ is $\tom$-definable in $\bf M$.
\end{proof}

\subsection{Stable metrics and theories}
We recall the following notion originating in the work of J.-L. Krivine and B. Maurey on stable Banach spaces \cite{KM}.
\begin{defi}
A metric $d$ on a set $X$ is said to be {\em stable} if,  for all $d$-bounded sequences $(x_n)$ and $(y_m)$ in $X$, we have
$$
\lim_{n\til \infty}\lim_{m\til \infty}d(x_n,y_m)=\lim_{m\til \infty}\lim_{n\til \infty}d(x_n,y_m),
$$
whenever both limits exist.
\end{defi}
We mention that stability of the metric is equivalent to requiring that the limit operations $\lim_{n\til \ku U}$ and $\lim_{m\til \ku V}$ commute over $d$ for all ultrafilters $\ku U$ and $\ku V$ on $\N$.

Now stability of metrics is tightly related to model theoretical stability of which we recall the definition.
\begin{defi}
Let $T$ be a complete theory of a countable language $\ku L$ and let $\kappa$ be an infinite cardinal number. We say that $T$ is {\em $\kappa$-stable} if, for all models ${\bf M}\models T$ and subsets $A\subseteq {\bf M}$ with $|A|\leqslant \kappa$, we have $|S^{\bf M}_n(A)|\leqslant \kappa$. Also, $T$ is {\em stable} if it is $\kappa$-stable for some infinite cardinal $\kappa$.
\end{defi}

In the following discussion, we shall always assume that $T$ is a complete theory with infinite models in a countable language $\ku L$. Of the various consequences of stability of $T$, the one most closely related to stability of metrics is the fact that, if $T$ is stable and ${\bf M}$ is a model of $T$, then there are no formula $\phi(\ov x, \ov y)$ and tuples $\ov a_n$, $\ov b_m$, $n,m\in \N$, so that 
$$
{\bf M}\models \phi(\ov a_n, \ov b_m)\;\equi \; n<m.
$$
Knowing this, the following lemma is straightforward.

\begin{lemme}\label{existence stable metric}
Suppose $\bf M$ is a countable atomic model of a stable theory $T$ and that  $\ov a$ is a finite tuple in $\bf M$. Let also $\rho$ be a metric on $\ku O(\ov a)$ so that, for every $n\in \N$, the relation $\rho(\ov b, \ov c)\leqslant n$ is $\tom$-definable in $\bf M$. Then $\rho$ is a stable metric. 
\end{lemme}

\begin{proof}
Suppose towards a contradiction that $\ov a_n, \ov b_m\in \ku O(\ov a)$ are bounded sequences in $\ku O(\ov a)$ so that
$$
r=\lim_{n\til \infty}\lim_{m\til \infty}\rho(\ov a_n,\ov b_m)\neq \lim_{m\til \infty}\lim_{n\til \infty}\rho(\ov a_n,\ov b_m)
$$
and pick a formula $\phi(\ov x, \ov y)$ so that 
$$
\rho(\ov b, \ov c)=r\;\;\equi\;\; {\bf M}\models \phi(\ov b, \ov c).
$$
Then, using $\a^\infty $ to denote ``for all, but finitely many'', we have
$$
\a ^\infty n\; \a^\infty m\;\;\; {\bf M}\models \phi(\ov a_n, \ov b_m),
$$
while 
$$
\a ^\infty m\; \a^\infty n\; \;\;{\bf M}\models \neg\phi(\ov a_n, \ov b_m).
$$
So, upon passing to subsequences of $(\ov a_n)$ and $(\ov b_m)$, we may suppose that 
$$
{\bf M}\models \phi(\ov a_n, \ov b_m)\;\;\equi\;\; n<m.
$$
However, the existence of such a formula $\phi$ and sequences $(\ov a_n)$ and $(\ov b_m)$ contradicts the stability of $T$.
\end{proof}

\begin{thm}
Suppose $\bf M$ is a countable atomic model of a stable theory $T$ so that ${\rm Aut}(\bf M)$ admits a maximal compatible left-invariant metric. Then  ${\rm Aut}(\bf M)$ admits a maximal compatible left-invariant metric, which, moreover, is stable. 
\end{thm}

\begin{proof}
By Theorem \ref{quasiisometry}, there is a finite tuple $\ov a$ and a finite family $\ku S$ of orbital types so that the mapping 
$$
g\in {\rm Aut}({\bf M})\mapsto g\ov a\in \mathbb X_{\ov a, \ku S}
$$
is a quasi-isometry of ${\rm Aut}({\bf M})$ with  $(\mathbb X_{\ov a, \ku S},  \rho_{\ov a, \ku S})$. Also, by Lemma \ref{existence stable metric}, $\rho_{\ov a, \ku S}$ is a stable metric on $\mathbb X_{\ov a, \ku S}$. Define also a compatible left-invariant stable metric $D\leqslant 1$ on ${\rm Aut}(\bf M)$ by
$$
D(g,f)=\sum_{n=1}^\infty \frac {\chi_{\neq}(g(b_n),f(b_n))} {2^{n}},
$$
where $(b_n)$ is an enumeration of $\bf M$ and $\chi_{\neq}$ is the characteristic funtion of inequality. The stability of  $D$ follows easily from it being an absolutely summable series of the functions $(g,f)\mapsto \frac {\chi_{\neq}(g(b_n),f(b_n))} {2^{n}}$.

Finally, let 
$$
d(g,f)=D(g,f)+\rho_{\ov a, \ku S}(g\ov a, f\ov a).
$$
Then $d$ is a maximal compatible left-invariant and stable metric on ${\rm Aut}(\bf M)$. 
\end{proof}

Similarly, when ${\rm Aut}(\bf M)$ is only assumed to have a metrically proper metric, this can also be taken to be stable. This can be done by working with the metric $\rho_{\ov a, (\ku S_n)}$, where $\ku S_1\subseteq \ku S_2\subseteq \ldots$ is an exhaustive sequence of finite sets of orbital types on $\bf M$, instead of $\rho_{\ov a, \ku S}$.

\begin{thm}
Suppose $\bf M$ is a countable atomic model of a stable theory $T$ so that ${\rm Aut}(\bf M)$ admits a  compatible metrically proper left-invariant metric. Then  ${\rm Aut}(\bf M)$ admits a compatible metrically proper left-invariant metric, which, moreover, is stable. 
\end{thm}

Using the equivalence of the local property (OB) and the existence of metrically proper metrics and that the existence of a metrically proper stable metric implies metrically proper actions on reflexive spaces, we have the following corollary.
\begin{cor}\label{atomic stable reflexive}
Suppose $\bf M$ is a countable atomic model of a stable theory $T$ so that ${\rm Aut}(\bf M)$ has the local property (OB). Then ${\rm Aut}(\bf M)$ admits a metrically proper continuous affine isometric action on a reflexive Banach space. 
\end{cor}

We should briefly consider the hypotheses of the preceding theorem. So, in the following, let $T$ be a complete theory with infinite models in a countable language $\ku L$. 
We recall that that  ${\bf M}\models T$ is said to be a {\em prime model} of $T$ if $\bf M$ admits an elementary embedding into every other model of $T$. By the omitting types theorem, prime models are necessarily atomic. In fact, ${\bf M}\models T$ is a prime model of $T$ if and only if ${\bf M}$ is both countable and atomic. Moreover, the theory $T$ admits a countable atomic model if and only if, for every $n$, the set of isolated types is dense in the type space $S_n(T)$. In particular, this happens if $S_n(T)$ is countable for all $n$.

Now, by definition,  $T$ is $\om$-stable, if, for every model ${\bf M}\models T$, countable subset $A\subseteq \bf M$ and $n\geqslant 1$, the type space $S_n^{\bf M}(A)$ is countable. In particular, $S_n(T)$ is countable for every $n$ and hence $T$ has a countable atomic model $\bf M$. Thus, provided that ${\rm Aut}(\bf M)$ has the local property (OB), Corollary \ref{atomic stable reflexive} gives a metrically proper affine isometric action of this automorphism group.

\subsection{Fra\"iss\'e classes}A useful tool in the study of ultrahomogeneous countable structures is the theory of R. Fra\"iss\'e that alllows us to view every such object as a so called limit of the family of its finitely generated substructures. In the following, we fix a countable language $\ku L$. 

\begin{defi}
A {\em Fra\"iss\'e class} is a class $\mathcal K$ of finitely generated $\ku L$-structures so that 
\begin{enumerate}
\item $\kappa$ contains only countably many isomorphism types,
\item (hereditary property) if ${\bf A}\in \mathcal K$ and $\bf B$ is a finitely generated $\ku L$-structure embeddable into ${\bf A}$, then ${\bf B}\in \mathcal K$,
\item (joint embedding property) for all ${\bf A}, {\bf B}\in \mathcal K$, there some ${\bf C}\in \mathcal K$ into which both ${\bf A}$ and ${\bf B}$ embed, 
\item (amalgamation property) if ${\bf A},{\bf B}_1, {\bf B}_2\in \mathcal K$ and $\eta_i\colon {\bf A}\inj {\bf B}_i$ are embeddings, then there is some ${\bf C}\in \mathcal K$ and embeddings $\zeta_i\colon {\bf B}_i\inj {\bf C}$ so that $\zeta_1\circ \eta_1=\zeta_2\circ\eta_2$.
\end{enumerate}
\end{defi}

Also, if ${\bf M}$ is a countable $\ku  L$-structure, we let ${\rm Age}(\bf M)$ denote the class of all finitely generated $\ku L$-structures embeddable into $\bf M$. 

The fundamental theorem of Fra\"iss\'e \cite{fraisse} states that, for every Fra\"iss\'e class $\mathcal K$, there is a unique (up to isomorphism) countable ultrahomogeneous structure $\bf K$, called the {\em Fra\"iss\'e limit} of $\mathcal K$, so that ${\rm Age}(\bf K)=\mathcal K$ and, conversely, if ${\bf M}$ is a countable ultrahomogeneous structure, then ${\rm Age}(\bf M)$ is a Fra\"iss\'e class. 

Now, if $\bf K$ is the limit of a Fra\"iss\'e class $\ku K$, then $\bf K$ is ultrahomogeneous and hence its orbital types correspond to quantifier-free types realised in $\bf K$. Now, as ${\rm Age}(\bf K)=\mathcal K$, for every quantifier-free type $p$ realised by some tuple $\ov a$ in ${\bf K}$, we see that the structure ${\bf A}=\langle\ov a\rangle$ generated by $\ov a$ belongs to $\ku K$ and that the expansion $\langle {\bf A}, \ov a\rangle$ of ${\bf A}$ with names for $\ov a$ codes $p$ by
$$
\phi(\ov x)\in p\;\equi \;\langle {\bf A}, \ov a\rangle\models \phi(\ov a).
$$
Vice versa, since ${\bf  A}$ is generated by $\ov a$, the quantifier free type ${\rm qftp}^{\bf A}(\ov a)$ fully determines the expanded structure $\langle{\bf A},\ov a\rangle$ up to isomorphism. To conclude, we see that orbital types $\ku O(\ov a)$ in $\bf K$ correspond to isomorphism types of expanded structures $\langle{\bf A},\ov a\rangle$, where $\ov a$ is a finite tuple generating some ${\bf A}\in \ku K$. This also means that Theorem \ref{quasiisometry} may be reformulated using these isomorphism types in place of orbital types. We leave the details to the reader and instead concentrate on a more restrictive setting.

Suppose now that ${\bf K}$ is the  limit of a Fra\"iss\'e class $\ku K$ consisting of finite structures, that is, if ${\bf K}$ is {\em locally finite}, meaning that every finitely generated substructure is finite. Then we note that every ${\bf A}\in \ku K$ can simply be enumerated by some finite tuple $\ov a$. Moreover, if ${\bf A}$ is a finite substructure of $\bf K$, then the pointwise stabiliser $V_{\bf A}$ is a finite index subgroup of the {\em setwise stabiliser}
$$
V_{\{\bf A\}}=\{g\in {\rm Aut}({\bf K})\del g{\bf A}={\bf A}\},
$$
so, in particular, $V_{\{\bf A\}}$ has property (OB) relative to ${\rm Aut}(\bf K)$ if and only if $V_{\bf A}$ does. Similarly, if ${\bf B}$ is another finite substructure, then  $V_{\bf A}$ is finitely generated over $V_{\{\bf B\}}$ if and only if it is finitely generated over $V_{\bf B}$. Finally, if $(\ov a, \ov c)\in \ku O(\ov a, \ov b)$ and ${\bf B}$ and ${\bf C}$ are the substructures of ${\bf K}$ generated by $(\ov a, \ov b)$ and $(\ov a, \ov c)$ respectively, then, for every automorphism $g\in {\rm Aut}(\bf K)$ mapping ${\bf B}$ to ${\bf C}$, there is an $h\in V_{\{\bf B\}}$ so that $gh(\ov a,\ov b)=(\ov a, \ov c)$.
Using these observations, one may easily modify the proof of  Theorem \ref{quasiisometry} to obtain the following variation.

\begin{thm}\label{quasiisometry locally finite}
Suppose  $\ku K$ is a Fra\"iss\'e class of finite structures with Fra\"iss\'e limit $\bf K$. Then ${\rm Aut}({\bf K})$ admits a maximal compatible left-invariant metric if and only if there is ${\bf A}\in\ku K$ satisfying the following two conditions.
\begin{enumerate}
\item For every  ${\bf B}\in \ku K$ containing ${\bf A}$, there are $n\geqslant 1$ and an isomorphism invariant family  $\ku S\subseteq \ku K$, containing only finitely many isomorphism types, so that, for all ${\bf C}\in \ku K$ and  embeddings $\eta_1,\eta_2\colon {\bf B}\inj {\bf C}$ with $\eta_1|_{\bf A}=\eta_2|_{\bf A}$, one can find some ${\bf D}\in \ku K$ containing ${\bf C}$ and a path ${\bf B}_0=\eta_1{\bf B}, {\bf B}_1, \ldots, {\bf B}_n=\eta_2{\bf B}$ of isomorphic copies of  ${\bf B}$ inside ${\bf D}$ with $\langle {\bf B}_i\cup {\bf B}_{i+1}\rangle\in \ku S$ for all $i$,
\item there is  an isomorphism invariant family $\ku R\subseteq \ku K$, containing only finitely many isomorphism types, so that, for all ${\bf B}\in \ku K$ containing ${\bf A}$ and isomorphic copies ${\bf A}'\subseteq \bf B$ of ${\bf A}$, there is some $\bf C\in \ku K$ containing $\bf B$ and a path ${\bf A}_0, {\bf A}_1, \ldots, {\bf A}_n\subseteq {\bf C}$ consisting of isomorphic copies of ${\bf A}$, beginning at ${\bf A}_0={\bf A}$ and ending at ${\bf A}_n={\bf A}'$, satisfying $\langle {\bf A}_i\cup {\bf A}_{i+1}\rangle\in \ku R$ for all $i$.
\end{enumerate}
\end{thm}

\section{Orbital independence relations}
The formulation of Theorem \ref{quasiisometry} is rather abstract and it is therefore useful to have some more familiar criteria for having the local property (OB) or having a well-defined quasi-isometry type. The first such criterion is simply a reformulation of an observation of P. Cameron.

\begin{prop}[P. Cameron]
Let $\bf M$ be an $\aleph_0$-categorical countable structure. Then, for every tuple $\ov a$ in $\bf M$, there is a finite set $F\subseteq {\rm Aut}(\bf M)$ so that ${\rm Aut}({\bf M})= V_{\ov a}FV_{\ov a}$. In particular, ${\rm Aut}(\bf M)$ has property (OB).
\end{prop}

\begin{proof}
Since $\bf M$ is $\aleph_0$-categorical, the pointwise stabiliser $V_{\ov a}$ induces only finitely many orbits on ${\bf M}^n$, where $n$ is the length of $\ov a$. So let $B\subseteq {\bf M}^n$ be a finite set of $V_{\ov a}$-orbit representatives. Also, for every $\ov b\in B$, pick if possible some $f\in {\rm Aut}(\bf M)$ so that $\ov b=f\ov a$ and let $F$ be the finite set of these $f$. Then, if $g\in {\rm Aut}(\bf M)$, as $g\ov a\in {\bf M}^n=V_{\ov a}B$, there is some $h\in V_{\ov a}$ and $\ov b\in B$ so that $g\ov a=h\ov b$. In particular, there is $f\in F$ so that $\ov b=f\ov a$, whence $g\ov a=h\ov b=hf\ov a$ and thus $g\in hfV_{\ov a}\subseteq V_{\ov a}FV_{\ov a}$.
\end{proof}

Thus, for an automorphism group ${\rm Aut}(\bf M)$ to have a non-trivial quasi-isometry type, the structure ${\bf M}$ should not be  $\aleph_0$-categorical. In this connection, we recall that if $\ku K$ is a Fra\"iss\'e class  in a finite language $\ku L$ and $\ku K$ is {\em uniformly locally finite}, that is, there is a function $f\colon \N\til \N$ so that every ${\bf A}\in \ku K$ generated by $n$ elements has size $\leqslant f(n)$, then the Fra\"iss\'e limit $\bf K$ is $\aleph_0$-categorical. In particular, this applies to Fra\"iss\'e classes of finite relational languages.

However, our first concern is to identify automorphism groups with the local property (OB) and, for this, we consider model theoretical independence relations.

\begin{defi}\label{indep rel}
Let $\bf M$ be a countable structure and $A\subseteq \bf M$ a finite subset. An {\em orbital $A$-independence relation} on $\bf M$ is a binary relation $\forkindep[A]$ defined between finite subsets of $\bf M$ so that, for all finite $B,C,D\subseteq \bf M$, 
\begin{itemize}
\item[(i)] (symmetry) $B\forkindep[A]C\equi C\forkindep[A]B$,
\item[(ii)] (monotonicity) $B\forkindep[A]C \;\;\&\;\; D\subseteq C\;\;\saa\;\; B\forkindep[A]D$, 
\item[(iii)] (existence) there is $f\in V_{A}$ so that ${fB}\forkindep[A]C$,
\item[(iv)] (stationarity) if $B\forkindep[A]C$ and $g\in V_A$ satisfies  $gB\forkindep[A]C$, then $g\in V_CV_B$, i.e., there is some $f\in V_C$ agreeing pointwise with $g$ on $B$.
\end{itemize}
\end{defi}
We read $B\forkindep[A]C$ as ``$B$ is independent from $C$ over $A$.'' Occasionally, it is convenient to let $\forkindep[A]$ be defined between finite tuples rather than sets, which is done by simply viewing a tuple as a name for the set it enumerates. For example, if $\ov b=(b_1,\ldots, b_n)$, we let $\ov b\forkindep[A]C$ if and only if $\{b_1,\ldots, b_n\}\forkindep[A]C$.

With this convention, the stationarity condition on $\forkindep[\ov a]$ can be reformulated as follows: If $\ov b$ and $\ov b'$ have the same orbital type over $\ov a$, i.e., $\OO(\ov b,\ov a)=\OO(\ov b', \ov a)$,  and are both independent from $\ov c$ over $\ov a$, then they also have the same orbital type over $\ov c$. 

Similarly, the existence condition on $\forkindep[\ov a]$ can be stated as: For all $\ov b, \ov c$, there is some $\ov b'$ independent from $\ov c$ over $\ov a$ and having the same orbital type over $\ov a$ as $\ov b$ does.

We should note that, as our interest is in the permutation group ${\rm Aut}(\bf M)$ and not the particular structure $\bf M$, any two structures $\bf M$ and ${\bf M}'$, possibly of different languages, having the same universe and the exact same automorphism group ${\rm Aut}({\bf M})={\rm Aut}({\bf M}')$ will essentially be equivalent for our purposes. We also remark that the existence of an orbital $A$-independence relation does not depend on the exact structure $\bf M$, but only on its universe and its automorphism group. Thus, in Examples \ref{measured boolean}, \ref{tree} and \ref{unitary}  below, any manner of formalising the mathematical structures  as bona fide first-order model theoretical structures of some language with the indicated automorphism group will lead to the same results and hence can safely be left to the reader.

\begin{exa}[Measured Boolean algebras]\label{measured boolean}
Let $\bf M$ denote the Boolean algebra of clopen subsets of Cantor space $\{0,1\}^\N$ equipped with the usual dyadic probability measure $\mu$, i.e., the infinite product of the $\{\frac 12, \frac 12\}$-distribution on $\{0,1\}$. We note that $\bf M$ is ultrahomogeneous, in the sense that, if $\sigma\colon A\til B$ is a measure preserving isomorphism between two subalgebras of $\bf M$, then $\sigma$ extends to a measure preserving automorphism of $\bf M$.   

For two finite subsets $A,B$, we let $A\forkindep[\tom]B$ if the Boolean algebras they generate are measure theoretically independent, i.e., if, for all $a_1,\ldots, a_n\in A$ and $b_1,\ldots, b_m\in B$, we have 
$$
\mu(a_1\cap \ldots\cap  a_n \cap b_1\cap \ldots\cap b_m)=\mu(a_1\cap \ldots\cap  a_n)\cdot \mu(b_1\cap \ldots\cap b_m).
$$ 
Remark that, if $\sigma\colon A_1\til A_2$ and $\eta\colon B_1\til B_2$ are measure preserving isomorphisms between subalgebras of $\bf M$ with $A_i\forkindep[\tom]B_i$, then there is a measure preserving isomorphism $\xi\colon \langle A_1\cup B_1\rangle\til    \langle A_2\cup B_2\rangle$ between the algebras generated extending both $\sigma$ and $\eta$. Namely, $\xi(a\cap b)=\sigma(a)\cap \eta(b)$ for atoms $a\in A$ and $b\in B$.

Using this and the ultrahomogeneity of $\bf M$, the stationarity condition (iv) of $\forkindep[\tom]$ is clear. Also,  symmetry and monotonicity are obvious. Finally, for the existence  condition (iii), suppose that $A$ and $B$ are given finite subsets of $\bf M$. Then there is some finite $n$ so that all elements of $A$ and $B$ can be written as unions of basic open sets $N_s=\{x\in \{0,1\}^\N\del s\text{ is an initial segment of }x\}$ for $s\in 2^n$.  Pick  a permutation  $\alpha$ of $\N$ so that $\alpha(i)>n$ for all $i\leqslant n$ and note that $\alpha$ induces  measure preserving automorphism $\sigma$ of $\bf M$ so that $\sigma(A)\forkindep[\tom]B$.

Thus, $\forkindep[\tom]$ is an orbital $\tom$-independence relation on $\bf M$. We also note that, by Stone duality, the automorphism group of $\bf M$ is isomorphic to the group ${\rm Homeo}(\{0,1\}^\N, \mu)$ of measure-preserving homeomorphisms of Cantor space.
\end{exa}

\begin{exa}[The ended $\aleph_0$-regular tree]\label{tree}
Let $\bf T$ denote the $\aleph_0$-regular tree. I.e., $\bf T$ is a countable connected undirected graph without loops in which every vertex has infinite valence. 
Since $\bf T$ is a tree, there is a natural notion of {\em convex hull}, namely, for a subset $A\subseteq \bf T$ and a vertex $x\in \bf T$, we set $x\in {\rm conv}(A)$ if there are $a,b\in A$ so that $x$ lies on the unique path from $a$ to $b$. Now, pick a distinguished vertex $t\in \bf T$ and, for finite $A,B\subseteq \bf T$, set 
$$
A\forkindep[{\{t\}}]B\;\equi\; {\rm conv}(A\cup\{t\})\cap {\rm conv}(B\cup\{t\})=\{t\}.
$$

That $\forkindep_{\{t\}}$ is both symmetric and monotone is obvious. Also, if $A$ and $B$ are finite, then so are ${\rm conv}(A\cup\{t\})$ and ${\rm conv}(B\cup\{t\})$ and so it is easy to find a ellitic isometry $g$ with fixed point $t$, i.e., a rotation of $\bf T$ around $t$, so that $g\big({\rm conv}(A\cup\{t\})\big)\cap {\rm conv}(B\cup\{t\})=\{t\}$. Since $g\big({\rm conv}(A\cup\{t\})\big)={\rm conv}(gA\cup\{t\})$, one sees that $gA\forkindep[\{t\}]B$, verifying the existence condition (iii). 

Finally, for the stationarity condition (iv), suppose $B,C\subseteq \bf T$ are given and $g$ is an elliptic isometry fixing $t$ so that $B\forkindep[{\{t\}}]C$ and $gB\forkindep[{\{t\}}]C$. Then, using again that $\bf T$ is $\aleph_0$-regular, it is easy to find another elliptic isometry fixing all of ${\rm conv}(C\cup \{t\})$ that agrees with $g$ on $B$.

So $\forkindep[\{t\}]$ is an orbital $\{t\}$-independence relation on $\bf T$.
\end{exa}

\begin{exa}[Unitary groups]\label{unitary}
Fix a countable field $\Q\subseteq {\mathfrak F}\subseteq \C$ closed under complex conjugation and square roots and let $\bf V$ denote the countable dimensional $ {\mathfrak F}$-vector space with basis $(e_i)_{i=1}^\infty$. We define the usual inner product on $\bf V$ by letting
$$
\Big\langle \sum_{i=1}^na_ie_i\Del \sum_{j=1}^mb_je_j\Big\rangle=\sum_i a_i\ov{b_i}
$$
and let $\ku U(\bf V)$ denote the corresponding unitary group, i.e., the group of all invertible linear transformations of $\bf V$ preserving $\langle\cdot\del \cdot \rangle$. 

For finite subsets $A,B\subseteq \bf V$, we let 
$$
A\forkindep[\tom]B\;\equi\; {\rm span}(A)\perp {\rm span}(B).
$$
I.e., $A$ and $B$ are independent whenever they span orthogonal subspaces. Symmetry and monotonicity is clear. Moreover, since we chose our field $\mathfrak F$ to be closed under complex conjugaction and square roots, the inner product of two vectors lies in $\mathfrak F$ and hence so does the norm $\norm{v}=\sqrt{\langle v\del v\rangle}$ of any vector. It follows that the Gram--Schmidt orthonormalisation procedure can be performed within $\bf  V$ and hence every orthonormal set may be extended to an orthonormal basis for $\bf V$. Using this, one may immitate the details of Example \ref{measured boolean} to show that $\forkindep[\tom]$ satisfies conditions (iii) and (iv). (See also Section 6 of \cite{OB} for additional details.)
\end{exa}

\begin{thm}\label{indep OB}
Suppose $\bf M$ is a countable structure, $A\subseteq \bf M$  a finite subset and $\forkindep[A]$ an orbital $A$-independence relation. Then the pointwise stabiliser subgroup $V_A$ has property (OB). Thus, if $A=\tom$, the automorphism group ${\rm Aut}(\bf M)$ has  property (OB) and, if $A\neq \tom$, ${\rm Aut}(\bf M)$ has the local property (OB). 
\end{thm}

\begin{proof}Suppose $U$ is an open neighbourhood of $1$ in $V_A$. We will find a finite subset $F\subseteq V_A$ so that $V_A=UFUFU$. By passing to a further subset, we may suppose that $U$ is of the form $V_B$, where $B\subseteq \bf M$ is a finite set containing $A$. We begin by choosing, using property (iii) of the orbital $A$-independence relation, some $f\in V_{A}$ so that ${fB}\forkindep[A]B$ and set $F=\{f,f\inv \}$.

Now, suppose that $g\in V_A$ is given and choose again by (iii) some $h\in V_{A}$ so that ${hB}\forkindep[A]{(B\cup gB)}$. By (ii), it follows that ${hB}\forkindep[A]{B}$ and ${hB}\forkindep[A]{gB}$, whereby, using (i), we have ${B}\forkindep[A]{hB}$ and ${gB}\forkindep[A]{hB}$. Since $g\in V_A$, we can apply (iv) to $C=hB$, whence $g\in V_{hB}V_B=hV_Bh\inv V_B$. 

However, as ${fB}\forkindep[A]B$ and ${hB}\forkindep[A]{B}$, i.e., ${(hf\inv \cdot fB)}\forkindep[A]{B}$, and also $hf\inv \in V_{A}$, by (iv) it follows that $hf\inv \in V_BV_{fB}=V_BfV_Bf\inv$. So, finally, $h \in V_BfV_B$ and 
$$
g\in hV_Bh\inv V_B\subseteq V_BfV_B\cdot V_B\cdot (V_BfV_B)\inv \cdot V_B  \subseteq   V_BFV_BFV_B
$$ 
as required.
\end{proof}
By the preceding examples, we see that both the automorphism group of the measured Boolean algebra and the unitary group $\ku (\bf V)$ have property (OB), while the automorphism group ${\rm Aut}(\bf T)$ has the local property (OB)  (cf. Theorem 6.20 \cite{turbulence}, Theorem 6.11 \cite{OB}, respectively Theorem 6.31 \cite{turbulence}).

We note that, if $\bf M$ is an $\omega$-homogeneous structure, $\ov a$, $\ov b$ are tuples in $\bf M$ and $A\subseteq {\bf M}$ is a finite subset, then, by definition, ${\rm tp}^{\bf M}(\ov a/A)={\rm tp}^{\bf M}(\ov b/A)$ if and only if $\ov b\in V_A\cdot \ov a$. 
In this case, we can reformulate conditions (iii) and (iv) of the definition of orbital $A$-independence relations as follows.
\begin{itemize}
\item[(iii)] For all $\ov a$ and $B$, there is $\ov b$ with ${\rm tp}^{\bf M}(\ov b/A)={\rm tp}^{\bf M}(\ov a/A)$ and $\ov b\forkindep[A]B$.
\item[(iv)] For all $\ov a, \ov b$ and $B$, if $\ov a\forkindep[A]B$, $\ov b\forkindep[A]B$ and  ${\rm tp}^{\bf M}(\ov a/A)={\rm tp}^{\bf M}(\ov b/A)$, then ${\rm tp}^{\bf M}(\ov a/B)={\rm tp}^{\bf M}(\ov b/B)$.
\end{itemize}

Also, for the next result, we remark that, if $T$ is a complete theory with infinite models in a countable language $\ku L$, then $T$ has a countable saturated model if and only if $S_n(T)$ is countable for all $n$. In particular, this holds if $T$ is $\omega$-stable.

\begin{thm}\label{stability}
Suppose that $\bf M$ is a saturated countable model of an $\om$-stable theory. Then ${\rm Aut}(\bf M)$ has property (OB).
\end{thm}

\begin{proof}
We note first that, since $\bf M$ is saturated and countable, it is $\om$-homogeneous. Now, since $\bf M$ is the model of an $\om$-stable theory, there is a corresponding notion of {\em forking independence} $\ov a\forkindep[A]B$ defined by
\[\begin{split}
\ov a\forkindep[A]B 
&\equi {\rm tp}^{\bf M}(\ov a/A\cup B) \text{ is a non-forking extension of }{\rm tp}^{\bf M}(\ov a/A)\\
&\equi {\rm RM}(\ov a/A\cup B)={\rm RM}(\ov a/A),
\end{split}\]
where $RM$ denotes the {\em Morley rank}.  In this case, forking independence $\forkindep[A]$ always satisfies symmetry and monotonicity, i.e., conditions (i) and (ii), for all finite $A\subseteq \bf M$. 

Moreover, by the existence of non-forking extensions,  every type ${\rm tp}^{\bf M}(\ov a/A)$ has a non-forking extension $q\in S_n(A\cup B)$. Also, as $\bf M$ is saturated, this extension $q$ is realised by some tuple $\ov b$ in $\bf M$, i.e., ${\rm tp}^{\bf M}(\ov b/A\cup B)=q$. Thus, ${\rm tp}^{\bf M}(\ov b/A\cup B)$ is a non-forking extension of  ${\rm tp}^{\bf M}(\ov a/A)={\rm tp}^{\bf M}(\ov b/A)$, which implies that ${\ov b}\forkindep[A]B$. 
In other words, for all for all $\ov a$ and $A,B$, there is $\ov b$ with ${\rm tp}^{\bf M}(\ov b/A)={\rm tp}^{\bf M}(\ov a/A)$ and $\ov b\forkindep[A]B$, which verifies the existence condition (iii) for $\forkindep[A]$.

However, forking independence over $A$, $\forkindep[A]$, may not satisfy the stationarity condition (iv) unless every type $S_n(A)$ is stationary, i.e., unless,  for all $B\supseteq A$, every type $p\in S_n(A)$ has a unique non-forking extension in $S_n(B)$. Nevertheless, as we shall show, we can get by with sligthly less.

We let $\forkindep[\tom]$ denote forking independence over the empty set. Suppose also that $B\subseteq \bf M$ is a fixed finite subset and let $\ov a\in {\bf M}^n$ be an enumeration of $B$. Then there are at most ${\rm deg}_M\big({\rm tp}^{\bf M}(\ov a)\big)$ non-forking extensions of ${\rm tp}(\ov a)$ in $S_n(B)$, where ${\rm deg}_M\big({\rm tp}^{\bf M}(\ov a)\big)$ denotes the {\em Morley degree} of ${\rm tp}^{\bf M}(\ov a)$. Choose realisations $\ov b_1,\ldots, \ov b_k\in {\bf M}^n$ for each of these non-forking extensions realised in $\bf M$. 
Since ${\rm tp}^{\bf M}(\ov b_i)={\rm tp}^{\bf M}(\ov a)$, there are $f_1,\ldots, f_k\in {\rm Aut}(\bf M)$ so that $\ov b_i=f_i\ov a$. Let $F$ be the set of these $f_i$ and their inverses. 
Thus, if $\ov c\in {\bf M}^n$ satisfies ${\rm tp}^{\bf M}(\ov c)={\rm tp}^{\bf M}(\ov a)$ and $\ov c\forkindep[\tom]B$, then there is some $i$ so that ${\rm tp}^{\bf M}(\ov c/B)={\rm tp}^{\bf M}(\ov b_i/B)$ and so, for some $h\in V_B$, we have $\ov c=h\ov b_i=hf_i\ov a\in V_BF\cdot \ov a$.

Now assume $g\in {\rm Aut}(\bf M)$ is given and pick, by condition (iii), some $h\in {\rm Aut}(\bf M)$ so that $hB\forkindep[\tom](B\cup gB)$. By monotonicity and symmetry, it follows that $B\forkindep[\tom]hB$ and $gB\forkindep[\tom]hB$. Also, since Morley rank  and hence forking independence are invariant under automorphisms of $\bf M$, we see that $h\inv B\forkindep[\tom]B$ and $h\inv gB\forkindep[\tom]B$. So, as $\ov a$ enumerates $B$, we have $h\inv \ov a\forkindep[\tom]B$ and $h\inv g\ov a\forkindep[\tom]B$, where clearly ${\rm tp}^{\bf M}(h\inv \ov a)={\rm tp}^{\bf M}(\ov a)$ and ${\rm tp}^{\bf M}(h\inv g\ov a)={\rm tp}^{\bf M}(\ov a)$. By our observation above, we deduce that $h\inv \ov a \in V_BF\cdot \ov a$ and $h\inv g\ov a \in V_BF\cdot \ov a$, whence $h\inv  \in V_BFV_{\ov a}=V_BFV_B$ and similarly $h\inv g \in V_BFV_B$. Therefore, we finally have that 
$$
g \in (V_BFV_B)\inv V_BFV_B=V_BFV_BFV_B.
$$
We have thus shown that, for all finite $B\subseteq \bf M$, there is a finite subset $F\subseteq {\rm Aut}(\bf M)$ so that ${\rm Aut}({\bf M})=V_BFV_BFV_B$, verifying that ${\rm Aut}(\bf M)$ has property (OB).
\end{proof}

\begin{exa}
Let us note that Theorem \ref{stability} fails without the assumption of $\bf M$ being saturated. To see this, let $T$ be the (complete) theory of $\aleph_0$-regular forests, i.e., of undirected graphs without loops in which every vertex has valence $\aleph_0$, in the language of a single binary edge relation. In all models of $T$, every connected component is then a copy of the $\aleph_0$-regular tree $\bf T$ and hence the number of connected components is a complete isomorphism  invariant for models of $T$. It follows, in particular, that the countable theory $T$ is $\aleph_1$-categorical and thus $\om$-stable (in fact, $\bf T$ is also $\om$-homogeneous). Nevertheless, the automorphism group of the unsaturated structure $\bf T$ fails to have property (OB) as witnessed by its tautological isometric action on $\bf T$.
\end{exa}

\begin{exa}
Suppose $\mathfrak F$ is a countable field and let $\ku L=\{+, -, 0\}\cup \{\lambda_t\del t\in \mathfrak F\}$ be the language of $\mathfrak F$-vector spaces, i.e., $+$ and $-$ are respectively  binary and unary function symbols and $0$ a constant  symbol representing the underlying Abelian group and $\lambda_t$ are unary function symbols representing multiplication by the scalar $t$. Let also $T$ be the theory of infinite $\mathfrak F$-vector spaces.

Since $\mathfrak F$-vector spaces of size $\aleph_1$ have dimension $\aleph_1$, we see that $T$ is $\aleph_1$-categorical and thus complete and $\omega$-stable. Moreover, provided $\mathfrak F$ is infinite, $T$ fails to be $\aleph_0$-categorical, since the formulas $x_1=\lambda_t(x_2)$, $t\in\mathfrak F$, give infinitely many different $2$-types. However, since $T$ is $\omega$-stable, it has a countable saturated model, which is easily seen to be the $\aleph_0$-dimensional $\mathfrak F$-vector space denoted $\bf V$. It thus follows from Theorem \ref{stability} that the general linear group ${\rm GL}(\bf V)={\rm Aut}(\bf V)$ has property (OB).
\end{exa}

Oftentimes, Fra\"iss\'e classes admit a canonical form of amalgamation that can be used to define a corresponding notion of independence. One rendering of this is given by K. Tent and M. Ziegler (Example 2.2 \cite{tent}). However, their notion is too weak to ensure that the corresponding independence notion is an orbital independence relation. For this, one needs a stronger form of  functoriality, which nevertheless is satisfied in most cases.

\begin{defi}\label{funct amal}
Suppose $\ku K$ is a Fra\"iss\'e class and that ${\bf A} \in \ku K$. We say that $\ku K$ admits a {\em functorial amalgamation} over $\bf A$ if there is map $\theta$ that to all pairs of embeddings $\eta_1\colon {\bf A}\inj {\bf B}_1$ and $\eta_2\colon {\bf A}\inj {\bf B}_2$, with ${\bf B}_1, {\bf B}_2\in\ku K$,  associates a pair of embeddings $\zeta_1\colon {\bf B}_1\inj \bf C$ and $\zeta_2\colon {\bf B}_2\inj \bf C$ into another structure ${\bf C}\in \ku K$ so that 
$\zeta_1\circ\eta_1=\zeta_2\circ\eta_2$ and, moreover, satisfying the following conditions.
\begin{enumerate}
\item(symmetry) The pair $\Theta\big(\eta_2\colon {\bf A}\inj {\bf B}_2, \eta_1\colon {\bf A}\inj {\bf B}_1\big)$ is the reverse of the pair $\Theta\big(\eta_1\colon {\bf A}\inj {\bf B}_1, \eta_2\colon {\bf A}\inj {\bf B}_2\big)$.
\item(functoriality) If $\eta_1\colon {\bf A}\inj {\bf B}_1$, $\eta_2\colon {\bf A}\inj {\bf B}_2$, $\eta'_1\colon {\bf A}\inj {\bf B}'_1$ and  $\eta'_2\colon {\bf A}\inj {\bf B}'_2$ are embeddings with ${\bf B}_1,{\bf B}_2, {\bf B}_1',{\bf B}_2'\in \ku K$ and $\iota_1\colon {\bf B}_1\inj {\bf B}_1'$ and $\iota_2\colon {\bf B}_2\inj {\bf B}_2'
$
are embeddings with $\iota_i\circ \eta_i=\eta_i'$, then, for 
$$
\Theta\big(\eta_1\colon {\bf A}\inj {\bf B}_1, \eta_2\colon {\bf A}\inj {\bf B}_2\big)=\big(\zeta_1\colon {\bf B}_1\inj {\bf C}, \zeta_2\colon {\bf B}_2\inj {\bf C}\big)
$$
and
$$
\Theta\big(\eta_1\colon {\bf A}\inj {\bf B}_1, \eta_2\colon {\bf A}\inj {\bf B}_2\big)=\big(\zeta'_1\colon {\bf B}'_1\inj  {\bf C}', \zeta'_2\colon {\bf B}'_2\inj  {\bf C}'\big),
$$
there is an embedding $\sigma\colon {\bf C}\inj{\bf C}'$ so that $\sigma\circ \zeta_i=\zeta_i'\circ\iota_i$ for $i=1,2$. 
\end{enumerate}
\end{defi}

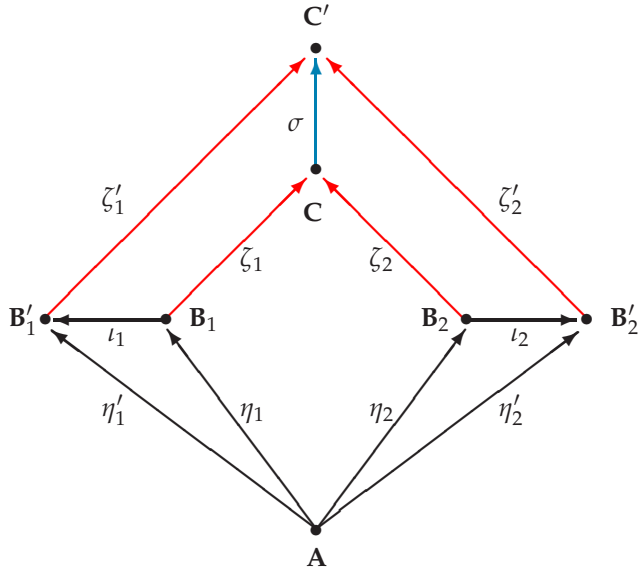
\setlength{\unitlength}{.4cm}
\begin{center}
\begin{picture}(0,18)

\thicklines
\put(0.2,0.2){\vector(3,4){5}}
\put(0.15,0.2){\vector(4,3){8.8}}
\put(0.25,0.25){\vector(-4,3){8.8}}

\put(0.3,0.2){\vector(-3,4){5}}

{\color{roed}          
\put(-4.8,7.2){\vector(1,1){4.7}}
\put(5.2,7.2){\vector(-1,1){4.7}}

\put(9.3,7.2){\vector(-1,1){8.7}}
\put(-8.85,7.2){\vector(1,1){8.7}}

}
\put(-4.8,7.2){\vector(-1,0){3.7}}
\put(5.2,7.2){\vector(1,0){3.7}}

{\color{groen}
\put(0.23,12.2){\vector(0,1){3.7}}
}

\put(-0.1,-1){${\bf A}$}
\put(-0.2,10.5){${\bf C}$}
\put(-0.2,17){${\bf C}'$}
\put(-4,7){${\bf B}_1$}
\put(-10,7){${\bf B}_1'$}
\put(3.7,7){${\bf B}_2$}
\put(10,7){${\bf B}_2'$}

\put(6.7,6.5){$\iota_2$}
\put(-6.7,6.5){$\iota_1$}

\put(-2.3,4){$\eta_1$}
\put(2,4){$\eta_2$}
\put(-6.9,4){$\eta'_1$}
\put(6.3,4){$\eta'_2$}

\put(-2.3,9){$\zeta_1$}
\put(2,9){$\zeta_2$}
\put(-6.9,11){$\zeta'_1$}
\put(6.3,11){$\zeta'_2$}

\put(-.7,13.5){$\sigma$}

\put(0,0){$\bullet$}
\put(0,12){$\bullet$}
\put(0,16){$\bullet$}
\put(-5,7){$\bullet$}
\put(5,7){$\bullet$}
\put(-9,7){$\bullet$}
\put(9,7){$\bullet$}

\end{picture}
\end{center}

We note that $\Theta\big(\eta_1\colon {\bf A}\inj {\bf B}_1, \eta_2\colon {\bf A}\inj {\bf B}_2\big)=\big(\zeta_1\colon {\bf B}_1\inj {\bf C}, \zeta_2\colon {\bf B}_2\inj {\bf C}\big)$ is simply the precise manner of decribing the amalgamation $\bf C$ of the two structures ${\bf B}_1$ and ${\bf B}_2$ over their common substructure ${\bf A}$ (with the additional diagram of embeddings). Thus, symmetry says that the amalgamation should not depend on the order of the structures ${\bf B}_1$ and ${\bf B}_2$, while functoriality states that the amalgamation should commute with embeddings of the ${\bf B}_i$ into larger structures ${\bf B}_i'$. With this concept at hand, for  finite subsets $A,B_1,B_2$ of the Fra\"iss\'e limit $\bf K$, we may define $B_1$ and $B_2$  to be independent over $A$ if ${\bf B}_1=\langle A\cup B_1\rangle$ and ${\bf B}_2=\langle A\cup B_2\rangle$ are amalgamated over ${\bf A}=\langle A\rangle$ in $\bf K$ as  given by $\Theta({\rm id}_{\bf A}\colon {\bf A}\inj{\bf B}_1, {\rm id}_{\bf A}\colon {\bf A}\inj{\bf B}_2)$. More precisely, we have the following definition.

\begin{defi}\label{functorial indep}
Suppose $\ku K$ is a Fra\"iss\'e class with limit $\bf K$, $A\subseteq\bf K$ is a finite subset and $\Theta$ is a functorial amalgamation on $\ku K$ over ${\bf A}=\langle A\rangle$. For finite subsets $B_1, B_2\subseteq \bf K$ with ${\bf B}_i=\langle A\cup B_i\rangle$, ${\bf D}=\langle A\cup B_1\cup B_2\rangle$ and 
$$
\Theta\big({\rm id}_{\bf A}\colon {\bf A}\inj{\bf B}_1, {\rm id}_{\bf A}\colon {\bf A}\inj{\bf B}_2\big)=\big(\zeta_1\colon {\bf B}_1\inj {\bf C}, \zeta_2\colon {\bf B}_2\inj {\bf C}\big),
$$
we set 
$$
B_1\forkindep[A]B_2
$$
if and only if there is an embedding $\pi\colon {\bf D}\inj {\bf C}$ so that $\zeta_i= \pi\circ {\rm id}_{{\bf B}_i}$ for $i=1,2$.
\end{defi}

With this setup, we readily obtain the following result.
\begin{thm}\label{functorial amalgamation OB}
Suppose $\ku K$ is a Fra\"iss\'e class with limit $\bf K$, $A\subseteq\bf K$ is a finite subset and $\Theta$ is a functorial amalgamation of $\ku K$ over ${\bf A}=\langle A\rangle$. Let also $\forkindep[A]$ be the relation defined from $\Theta$ and $A$ as in Definition \ref{functorial indep}. Then $\forkindep[A]$  is an orbital $A$-independence relation on $\bf K$ and the open subgroup $V_A$ has property (OB) relative to ${\rm Aut}(\bf K)$. In particular, ${\rm Aut}(\bf K)$ admits a  metrically proper compatible left-invariant metric.
\end{thm}

\begin{proof}Symmetry and monotonicity of $\forkindep[A]$ follow easily from  symmetry, respectively functoriality, of $\Theta$. Also, the existence condition on $\forkindep[A]$ follows from the ultrahomogeneity of $\bf K$ and the realisation of the amalgam $\Theta$ inside of $\bf K$. 

For stationarity, we use the ultrahomogeneity of $\bf K$. So, suppose that finite $\ov a, \ov b$ and $B\subseteq \bf K$ are given so that $\ov a\forkindep[A]B$, $\ov b\forkindep[A]B$ and ${\rm tp}^{\bf K}(\ov a/A)={\rm tp}^{\bf K}(\ov b/A)$.  We set ${\bf B}_1=\langle \ov a\cup A\rangle$, ${\bf B}'_1=\langle \ov b\cup A\rangle$, ${\bf B}_2=\langle B\cup A\rangle$, ${\bf D}=\langle \ov a\cup B\cup A\rangle$ and ${\bf D}'=\langle \ov b\cup B\cup A\rangle$. Let also
$$
\Theta\big({\rm id}_{\bf A}\colon {\bf A}\inj {\bf B}_1, {\rm id}_{\bf A}\colon {\bf A}\inj {\bf B}_2\big)=
\big(\zeta_1\colon {\bf B}_1\inj {\bf C},\zeta_2\colon {\bf B}_2\inj {\bf C}\big),
$$
$$
\Theta\big({\rm id}_{\bf A}\colon {\bf A}\inj {\bf B}'_1, {\rm id}_{\bf A}\colon {\bf A}\inj {\bf B}_2\big)=
\big(\zeta'_1\colon {\bf B}'_1\inj {\bf C}',\zeta'_2\colon {\bf B}_2\inj {\bf C}'\big)
$$
and note that there is an isomorphism $\iota\colon {\bf B}_1\inj {\bf B}_1'$ pointwise fixing $A$  so that $\iota(\ov a)=\ov b$. By the definition of the independence relation, there are embeddings $\pi\colon {\bf D}\inj {\bf C}$ and $\pi'\colon {\bf D}'\inj {\bf C}'$ so that $\zeta_i=\pi\circ {\rm id}_{{\bf B}_i}$, $\zeta'_1=\pi'\circ {\rm id}_{{\bf B}'_1}$ and $\zeta_2'=\pi'\circ {\rm id}_{{\bf B}_2}$. On the other hand, by the functoriality of $\Theta$, there is an embedding $\sigma\colon {\bf C}\inj {\bf C}'$ so that $\sigma\circ \zeta_1=\zeta_1'\circ \iota$ and $\sigma\circ \zeta_2=\zeta_2'\circ {\rm id}_{{\bf B}_2}$. Thus, 
$\sigma\circ \pi\circ {\rm id}_{{\bf B}_1}=\pi'\circ {\rm id}_{{\bf B}_1'}\circ \iota$ and $\sigma\circ \pi\circ {\rm id}_{{\bf B}_2}=\pi'\circ {\rm id}_{{\bf B}_2}\circ{\rm id}_{{\bf B}_2}$, i.e., $\sigma\circ \pi|_{{\bf B}_1}=\pi'\circ \iota$ and  $\sigma\circ \pi|_{{\bf B}_2}=\pi'|_{{\bf B}_2}$. Let now $\rho\colon \pi'[{\bf D}']\inj {\bf D}'$ be the isomorphism that is inverse to $\pi'$. Then $\rho\sigma\pi|_{{\bf B}_1}=\iota$ and $\rho\sigma\pi|_{{\bf B}_2}={\rm id}_{{\bf B}_2}$. So, by ultrahomogeneity of $\bf K$, there is an automorphism $g\in {\rm Aut}(\bf K)$ extending $\rho\sigma\pi$, whence, in particular, $g\in V_{{\bf B}_2}\subseteq V_B$, while $g(\ov a)=\ov b$. It follows that ${\rm tp}^{\bf K}(\ov a/B)={\rm tp}^{\bf K}(\ov b/B)$, verifying stationarity.
\end{proof}

\begin{exa}[Urysohn metric spaces, cf. Example 2.2 (c) \cite{tent}]\label{urysohn}
Suppose $\ku S$ is a countable additive subsemigroup of the positive reals. Then the class of finite metric spaces with distances in $\ku S$ forms a Fra\"iss\'e class $\ku K$ with functorial amalgamation over the one-point metric space ${\bf P}=\{p\}$. Indeed, if $\bf A$ and $\bf B$ belong to $\ku S$ and intersect exactly in the point $p$, we can define a metric $d$ on ${\bf A}\cup {\bf B}$ extending those of ${\bf A}$ and ${\bf B}$ by letting
$$
d(a,b)=d_{\bf A}(a,p)+d_{\bf B}(p,b),
$$
for $a\in \bf A$ and $b\in \bf B$.
We thus take this to define the amalgamation of $\bf A$ and $\bf B$ over ${\bf P}$ and one easily verifies that this provides a functorial amalgamation over $\bf P$ on the class $\ku K$.

Two important particular cases are when $\ku S=\Z_+$, respectively $\ku S=\Q_+$, in which case the Fra\"iss\'e limits are the integer and rational Urysohn metric spaces $\Z\U$ and $\Q\U$. By Theorem \ref{functorial amalgamation OB}, we see that their isometry groups ${\rm Isom}(\Z\U)$ and ${\rm Isom}(\Q\U)$ admit compatible metrically proper left-invariant metrics.

We note also that it is vital that $\bf P$ is non-empty. Indeed, since  ${\rm Isom}(\Z\U)$ and ${\rm Isom}(\Q\U)$ act transitively on metric spaces of infinite diameter, namely,  on $\Z\U$ and $\Q\U$, they do not have property (OB) and hence the corresponding Fra\"iss\'e classes do not admit a functorial amalgamation over the empty space $\tom$.

Instead, if, for a given $\ku S$ and $r\in \ku S$, we let $\ku K$ denote the finite metric spaces with distances in $\ku S\cap [0,r]$, then $\ku  K$ is still a Fra\"iss\'e class now admitting functorial amalgamation over the empty space. Namely, to join $\bf A$ and $\bf B$, one simply takes the disjoint union and stipulates that $d(a,b)=r$ for all $a\in \bf A$ and $b\in \bf B$. 

As a particular example, we note that the isometry group ${\rm Isom}(\Q\U_1)$ of the rational Urysohn metric space of diameter $1$ has property (OB) (Theorem 5.8 \cite{OB}).
\end{exa}

\begin{exa}[The ended $\aleph_0$-regular tree]\label{ended tree}
Let again $\bf T$ denote the $\aleph_0$-regular tree and fix an {\em end} $\mathfrak e$ of $\bf T$. That is, $\mathfrak e$ is an equivalence class of infinite paths $(v_0, v_1, v_2, \ldots)$ in $\bf T$ under the equivalence relation
$$
(v_0, v_1, v_2, \ldots)\sim (w_0, w_1, w_2, \ldots)\;\equi \; \e k,l\;\a n \;v_{k+n}=w_{l+n}.
$$
So, for every vertex $t\in \bf T$, there is a unique path $(v_0, v_1, v_2, \ldots)\in \mathfrak e$ beginning at $v_0=t$. Thus, if $r$ is another vertex in $\bf T$, we can set $t<_{\mathfrak e}r$ if and only if $r=v_n$ for some $n\geqslant 1$. Note that this defines a strict partial ordering $<_{\mathfrak e}$ on $\bf T$ so that every two vertices $t,s\in \bf T$ have a least upper bound and, moreover, this least upper bound lies on the geodesic from $t$ to $s$. Furthermore, we define the function $\vartheta\colon {\bf T}\times {\bf T}\til {\bf T}$ by letting $\vartheta(t,s)=x_1$, where $(x_0,x_1,x_2,\ldots,x_k)$ is the geodesic from $x_0=t$ to $x_k=s$, for $t\neq s$,  and $\vartheta(t,t)=t$.

As is easy to see, the expanded structure $\big({\bf T}, <_{\mathfrak e}, \vartheta\big)$ is ultrahomogeneous and locally finite and hence, by Fra\"iss\'e's Theorem, is the Fra\"iss\'e limit of its age $\ku K={\rm Age}({\bf T}, <_{\mathfrak e}, \vartheta)$. We also claim that $\ku K$ admits a functorial amalgamation over the structure on a single vertex $t$. Indeed, if ${\bf A}$ and ${\bf B}$ are finite subtructures of $\big({\bf T}, <_{\mathfrak e}, \vartheta\big)$ and we pick a vertex in $t_{\bf A}$ and $t_{\bf B}$ in each, then there is a freest amalgamation of ${\bf A}$ and ${\bf B}$ identifying $t_{\bf A}$ and $t_{\bf B}$. Namely, let $t_{\bf A}=a_0<_{\mathfrak e}a_1<_{\mathfrak e}a_2<_{\mathfrak e}\ldots <_{\mathfrak e}a_n$ and $t_{\bf B}=b_0<_{\mathfrak e}b_1<_{\mathfrak e}b_2<_{\mathfrak e}\ldots <_{\mathfrak e}b_m$ be an enumeration of the successors of $t_{\bf A}$ and $t_{\bf B}$ in ${\bf A}$ and $\bf B$ respectively. We then take the disjoint union of ${\bf A}$ and ${\bf B}$ modulo the identifications $a_0=b_0, \; \ldots\;, a_{\min(n,m)}=b_{\min(n,m)}$ and add only the edges from ${\bf A}$ and ${\bf B}$. There are then unique extensions of $<_{\mathfrak e}$ and $\vartheta$ to the amalgam making it a member of $\ku K$. Moreover, this amalgamation is functorial over the single vertex $t$.

It thus follows that ${\rm Aut}({\bf T}, <_{\mathfrak e}, \vartheta)$ has the local property (OB) as witnessed by the pointwise stabiliser $V_t$ of any fixed vertex $t\in \bf T$. Now, as $\vartheta$ commutes with automorphisms of ${\bf T}$ and $<_{\mathfrak e}$ and $\mathfrak e$ are interdefinable, we see that ${\rm Aut}({\bf T}, <_{\mathfrak e}, \vartheta)$ is simply the group ${\rm Aut}({\bf T}, \mathfrak e)$ of all automorphisms of ${\bf T}$ fixing the end $\mathfrak e$. 
\end{exa}


\section{Computing quasi-isometry types of automorphism groups}
Thus far we have been able to show that certain automorphism groups have the local property (OB) and hence a compatible metrically proper left-invariant metric. The goal is now to identity their quasi-isometry type insofar as this is well-defined. 

\begin{exa}[The $\aleph_0$-regular tree]\label{tree 2}
Let $\bf T$ be the $\aleph_0$-regular tree and fix a vertex $t\in \bf T$.  By Example \ref{tree} and Theorem \ref{indep OB}, we know that $V_t$ has property (OB) relative to ${\rm Aut}(\bf T)$. Fix also a neighbour $s$ of $t$ in $\bf T$ and let $\ku R=\{\ku O(t,s)\}$. Now, since ${\rm Aut}(\bf T)$ acts transitively on the set of oriented  edges of $\bf T$, we see that, if $r\in \ku O(t)=\bf T$ is any vertex and $(v_0,v_1, \ldots, v_m)$ is the geodesic from  $v_0=t$ to $v_m=r$, then $\ku O(v_i, v_{i+1})\in\ku R$ for all $i$. It thus follows from Theorem \ref{quasiisometry} that ${\rm Aut}(\bf T)$ admits a maximal compatible left-invariant metric and, moreover, that
$$
g\in {\rm Aut}({\bf T})\mapsto g(t)\in \mathbb X_{t, \ku R}
$$
is a quasi-isometry. However, the graph $\mathbb X_{t, \ku R}$ is simply the tree $\bf T$ itself, which shows that
$$
g\in {\rm Aut}({\bf T})\mapsto g(t)\in\bf T
$$
is a quasi-isometry. In other words, the quasi-isometry type of ${\rm Aut}(\bf T)$ is just the tree ${\bf T}$.
\end{exa}

\begin{exa}[The ended $\aleph_0$-regular tree] 
Let $\big({\bf T}, <_{\mathfrak e}, \vartheta\big)$ be as in Example \ref{ended tree}. Again, if $t$ is some fixed vertex, then $V_t$ has property (OB) relative to  ${\rm Aut}({\bf T}, <_{\mathfrak e}, \vartheta)={\rm Aut}({\bf T}, \mathfrak e)$. Now, ${\rm Aut}({\bf T}, <_{\mathfrak e}, \vartheta)$ acts transitively on the vertices and edges of ${\bf T}$, but no longer acts transitively on set of oriented edges. Namely, if $(x,y)$ and $(v,w)$ are edges of $\bf T$, then $(v,w)\in \ku O(x,y)$ if and only if $x<_{\mathfrak e}y\leftrightarrow v<_{\mathfrak e}w$. Therefore, let $s$ be any neighbour of $t$ in $\bf T$ and set $\ku R=\{\ku O(t,s),\ku O(s,t)\}$. Then, as in Example \ref{tree 2}, we see that $\mathbb X_{t, \ku R}=\bf T$ and that
$$
g\in {\rm Aut}({\bf T},\mathfrak e)\mapsto g(t)\in  \bf T
$$
is a quasi-isometry. So  $ {\rm Aut}({\bf T},\mathfrak e)$ is quasi-isometric to $\bf T$ and thus also to $ {\rm Aut}({\bf T})$.
\end{exa}

\begin{exa}[Urysohn metric spaces]
Let $\ku S$ be a countable additive subsemigroup of the positive reals and let $\ku S\U$ be the limit of the Fra\"iss\'e class $\ku K$ of finite metric spaces with distances in $\ku S$. As we have seen in Example \ref{urysohn}, $\ku K$ admits a functorial amalgamation over the one point metric space ${\bf P}=\{p\}$ and thus the isometry group ${\rm Isom}(\ku S\U)$ with the permutation group topology has the local property (OB) as witnessed by the stabiliser $V_{x_0}$ of any chosen point $x_0\in \ku S\U$. 

We remark that $\ku O(x_0)=\ku S\U$ and fix some point $x_1\in \ku S\U\setminus \{x_0\}$ and $s=d(x_0,x_1)$. Let also $\ku R=\{\OO(x_0,x_1)\}$. By the ultrahomogeneity of $\ku S\U$, we see that $\ku O(y,z)=\ku O(x_0, x_1)\in \ku R$ for all  $y,z\in \ku S\U$ with $d(y,z)=s$. 

Now, for any two points $y,z\in \ku S\U$, let $n_{y,z}=\lceil{\frac {d(y,z)}s}\rceil+1$.  It is then easy to see that there is a finite metric space in $\ku K$ containing a sequence of points $v_0, v_1, \ldots, v_{n_{y,z}}$ so that $d(v_i,v_{i+1})=s$, while $d(v_0, v_{n_{y,z}})=d(y,z)$. By the ultrahomogeneity of $\ku S\U$, it follows that there is a sequence $w_0, w_1, \ldots, w_{n_{y,z}}\in \ku S\U$ with $w_0=y$, $w_{n_{y,z}}=z$ and $d(w_i, w_{i+1})=s$, i.e., $\ku O(w_i,w_{i+1})\in \ku R$ for all $i$. In other words, 
$$
\rho_{x_o, \ku R}(y,z)\leqslant \frac 1sd(y,z)+2
$$
and, in particular, the graph $\mathbb X_{x_0,\ku R}$ is connected.
Conversely, if $\rho_{x_0, \ku R}(y,z)=m$, then there is a finite path $w_0, w_1, \ldots, w_m\in \ku S\U$ with  $w_0=y$, $w_m=z$ and $d(w_i, w_{i+1})=s$, whereby $d(y,z)\leqslant ms$, showing that
$$
\frac 1sd(y,z)\leqslant \rho_{x_o, \ku R}(y,z)\leqslant \frac 1sd(y,z)+2.
$$
Therefore, the identity map is a quasi-isometry between $\mathbb X_{x_0,\ku R}$ and $\ku S\U$. Since, by Theorem \ref{quasiisometry}, the mapping
$$
g\in {\rm Isom}(\ku S\U)\mapsto g(x_0)\in \mathbb X_{x_0,\ku R}
$$
is a quasi-isometry, so is the mapping
$$
g\in {\rm Isom}(\ku S\U)\mapsto g(x_0)\in \ku S\U.
$$
By consequence, the isometry group $ {\rm Isom}(\ku S\U)$ is quasi-isometric to the Urysohn space $\ku S\U$. 
\end{exa}




\begin{thebibliography}{99}

\bibitem{abels} H. Abels, {\em Specker-Kompaktifizierungen von lokal kompakten topologischen Gruppen}, Math. Z. 135 (1973/74), 325--361.

\bibitem{wap}I. Ben Yaacov and T. Tsankov, {\em Weakly almost periodic functions, model theoretic stability, and minimality of topological groups}, 	arXiv:1312.7757.

\bibitem{fraisse}R. Fra\"iss\'e, {\em Sur l'extension aux relations de quelques propriet\'es des ordres}, Ann. Sci. \'Ecole Norm. Sup. 71 (1954), 363--388.

\bibitem{hodges}W. Hodges, {\em A shorter model theory}, Cambridge University Press, Cambridge (1997).

\bibitem{kpt} A.S. Kechris, V.G. Pestov and S. Todorcevic, {\em Fra\"iss\'e
limits, Ramsey theory, and topological dynamics of automorphism
groups},   Geom. Funct. Anal. 15  (2005),  no. 1, 106--189.


\bibitem{turbulence}A. S. Kechris and Christian Rosendal, {\em Turbulence, amalgamation and generic automorphisms of homogeneous structures}, Proc. Lond. Math. Soc. (3) 94 (2007), no. 2, 302--350.

\bibitem{KM}J.-L. Krivine and B. Maurey, {\em Espaces de Banach stables}, Israel J. Math.   39 (1981), no. 4, 273--295.

\bibitem{marker}D. E. Marker, {\em Model theory: An introduction}, Springer-Verlag, New York (2002).


\bibitem{OB}C. Rosendal, {\em A topological version of the Bergman property}, Forum Mathematicum 21 (2009), no. 2, 299--332.



\bibitem{large scale geom}C. Rosendal, {\em Large scale geometry of metrisable groups}, preprint.

\bibitem{tent}K. Tent and M. Ziegler, {\em On the isometry group of the Urysohn space}, J. London Math. Soc. (2) 87 (2013) 289--303.

\bibitem{tsankov}T. Tsankov, {\em Unitary representations of oligomorphic groups}, Geom. Funct. Anal. 22 (2012), no. 2, 528--555.


\end{thebibliography}
\end{document}